\documentclass[12pt]{article}
\usepackage{amsthm,amsmath,amsfonts,amssymb}
\usepackage{tikz}
\usetikzlibrary{calc}

\usepackage[colorlinks=true,citecolor=black,linkcolor=black,urlcolor=blue]{hyperref}

\theoremstyle{plain}
\newtheorem{theorem}{Theorem}
\newtheorem{lemma}[theorem]{Lemma}

\newtheorem{proposition}[theorem]{Proposition}

\theoremstyle{definition}

\newtheorem{example}[theorem]{Example}
\newtheorem{conjecture}[theorem]{Conjecture}

\theoremstyle{remark}
\newtheorem{remark}[theorem]{Remark}

\title{\bf The $r$-matching sequencibility of complete graphs}

\author{Adam Mammoliti \\
\small School of Mathematics and Statistics\\[-0.8ex]
\small UNSW Sydney\\[-0.8ex]
\small NSW 2052, Australia\\
\small\tt a.mammoliti@unsw.edu.au\\
}

\date{
}

\begin{document}

\maketitle

\begin{abstract}
Alspach [{\em Bull. Inst. Combin. Appl.}, 52 (2008), pp. 7--20] defined the maximal matching sequencibility of a graph $G$,
denoted~$ms(G)$,
to be the largest integer $s$ for which
there is an ordering of the edges of $G$ such that every $s$ consecutive edges form a matching.
Alspach also proved that $ms(K_n) = \bigl\lfloor\frac{n-1}{2}\bigr\rfloor$.
Brualdi et al. [{\em Australas. J. Combin.}, 53 (2012), pp. 245--256] extended the definition to cyclic matching sequencibility of a graph $G$,
denoted $cms(G)$,
which allows cyclical orderings
and proved that $cms(K_n) = \bigl\lfloor\frac{n-2}{2}\bigr\rfloor$.

In this paper, we generalise these definitions to require that
every $s$ consecutive edges form a subgraph where every vertex has degree at most $r\geq 1$,
and we denote the maximum such number for a graph $G$ by $ms_r(G)$
and $cms_r(G)$ for the non-cyclic and cyclic cases, respectively.
We conjecture that $ms_r(K_n) = \bigl\lfloor\frac{rn-1}{2}\bigr\rfloor$
and ${\bigl\lfloor\frac{rn-1}{2}\bigr\rfloor-1}~
\leq cms_r(K_n)  \leq  \bigl\lfloor\frac{rn-1}{2}\bigr\rfloor$
and that both bounds are attained for some $r$ and~$n$.
We prove these conjectured identities for the majority of cases,
by defining and characterising selected decompositions of $K_n$.
We also provide bounds on $ms_r(G)$ and $cms_r(G)$
as well as results on hypergraph analogues of $ms_r(G)$ and $cms_r(G)$.

\end{abstract}

\section{Introduction}

The (maximal) {\em matching sequencibility} of a simple graph $G$,
denoted $ms(G)$,
is the largest integer $s$ for which
there exists an ordering of the edges of $G$ so that every $s$ consecutive edges form a matching.
Alspach~\cite{MR2394738} determined  $ms(K_n)$,
as follows.

\begin{theorem}[Alspach \cite{MR2394738}]\label{thm:Matching sequence}
For each integer $n \geq 3$,
\[
ms(K_n) = \left\lfloor\frac{n-1}{2}\right\rfloor\,.
\]
\end{theorem}

Brualdi, Kiernan, Meyer and Schroeder~\cite{MR2961987}
considered the {\em cyclic matching sequencibility}
$cms(G)$ of a graph $G$, which is the natural analogue of the matching sequencibility
for $G$ when cyclic orders are allowed.
They proved the cyclic analogue of Theorem~\ref{thm:Matching sequence}, below.

\begin{theorem}[Brualdi et al.~\cite{MR2961987}]\label{thm:Cyclic matching sequence}
For each integer $n \geq 4$,
\[
cms(K_n) = \left\lfloor\frac{n-2}{2}\right\rfloor\,.
\]
\end{theorem}

The aim of this paper is to extend
Theorem~\ref{thm:Matching sequence} and
Theorem~\ref{thm:Cyclic matching sequence}
by generalising the notion of matching sequencibility.
In particular, for a graph $G$, $ms_r(G)$ denotes the analogue
of $ms(G)$ where consecutive edges form a subgraph
whose vertices each has degree at most~$r$.
Similarly, $cms_r(G)$ is defined analogously to
$ms_r(G)$ where we allow cyclic orderings of the edges of $G$.

\begin{conjecture}\label{conj: Kn}
Let $n \geq 3$ and $1 \leq r \leq n-2$ be integers.
Then
\[
ms_r(K_n) = \left\lfloor\frac{rn-1}{2}\right\rfloor
\quad \textrm{and} \quad
      \left\lfloor\frac{rn-1}{2}\right\rfloor - 1
 \leq cms_r(K_n)
 \leq \left\lfloor\frac{rn-1}{2}\right\rfloor\,.
\]
\end{conjecture}

The main results include the three to follow which verify
the conjecture in many cases. In each result we assume $n \geq 3$
and $1 \leq r \leq n-2$.

\begin{theorem}\label{thm: Matching sequity for r}
If $n$ or $r$ is even,
or $n$ is odd and either $r\geq \frac{n-1}{2}$ or $\gcd(r,n-1) = 1$,
then
\[
  ms_r(K_n) = \left\lfloor\frac{rn-1}{2}\right\rfloor\,.
\]
\end{theorem}
\begin{theorem}\label{thm: Cyclic matching sequity for r}
If $n$ is even, or $n$ is odd and $r = \frac{n-1}{2}$,
then
\[
  cms_r(K_n) = \left\lfloor\frac{rn-1}{2}\right\rfloor\,.
\]
\end{theorem}

\begin{theorem}\label{thm: Gen cyclic matching sequence odd n}
If $n$ is odd and $r$ is even, then
\[
       \left\lfloor\frac{rn-1}{2}\right\rfloor-1
  \leq cms_r(K_n)
  \leq \left\lfloor\frac{rn-1}{2}\right\rfloor\,.
\]
\end{theorem}

One might ask which of the above bounds holds for which values of $r$ and $n$.
We discuss this question at the end of the paper and prove the following theorem
which is the fourth and final of our main results.

\begin{theorem}\label{thm: K_n comp equiv}
For odd integers $r$ and $n$,
\[
  cms_r(K_n)       = \left\lfloor\frac{rn-1}{2}\right\rfloor \quad \textrm{if and only if }\quad
  cms_{n-1-r}(K_n) = \left\lfloor\frac{(n-1-r)n-1}{2}\right\rfloor \,.
\]
\end{theorem}

The paper is organised as follows.
In Section \ref{Preliminaries}
we generalise the methods of
\cite{MR2394738} and \cite{MR2961987},
expressed as Propositions
\ref{prop: Matching decomposition both}--\ref{prop: 2-regular decomposition odd r both}.
The propositions allow us to reduce the problem
of determining $ms_r(G)$ and $cms_r(G)$
to ordering subgraphs of $G$ which are partially
$r$-sequenceable for a smaller value of $r$.
However, the parities of $n$ and $r$ play
a crucial role in the effectiveness of
Propositions
\ref{prop: Matching decomposition both}--\ref{prop: 2-regular decomposition odd r both}:
the case when $n$ is odd is trickier and more so
when $r$ is also odd.

Section~\ref{sec:decompositions-of-Kn} defines
the {\em Walecki decomposition}~\cite{MR2394738} and other decompositions of $K_n$.
These are central to the proofs of Theorems~\ref{thm: Matching sequity for r}--\ref{thm: Gen cyclic matching sequence odd n};
those are presented in Sections \ref{sec: Gen matching sequence}--\ref{sec: r matching seq odd r centre case}.
Section~\ref{sec: Gen} presents the proof of Theorem \ref{thm: K_n comp equiv}
and, as part of that proof,
we consider sequencibility when certain general conditions
are placed on consecutive edges of orderings of graphs.
Section~\ref{sec: Con} concludes the paper with a discussion on Conjecture~\ref{conj: Kn} and related open problems,
and we provide some recursive bounds on $ms_r(G)$ and $cms_r(G)$ for general graphs $G$
as well as for the complete $k$-graph $\mathcal{K}^k_n$;
see Proposition~\ref{prop: 1 to r} and Theorem~\ref{thm:Katonaplus}, respectively.

\section{Preliminaries}
\label{Preliminaries}

\noindent
In this paper, graphs will always be simple.
A matching of a graph $G$ is a subgraph $M$ in which each vertex has degree~$1$.
A graph $G$ is {\em $( \leq r)$-regular} if each of its vertices has degree at most~$r$.
If every vertex has degree equal to $r$, then $G$ is {\em $r$-regular}.
In particular, a matching of a graph is a $1$-regular subgraph.
For an integer $n$, let $[n]:= \{0,1, \ldots , n-1 \}$, where $[0] = \emptyset$.
An {\em ordering} or {\em labelling} of a graph $G = (V,E)$ is a bijective function
$\ell \;:\; E \rightarrow [|E|]$.
The image of $e$ under~$\ell$ is called the {\em label} of $e$.
The edges $e_0 , \ldots , e_{s-1}$ are {\em consecutive} in $\ell$
if the labels of $e_0 , \ldots , e_{s-1}$ are consecutive integers.
For an ordering $\ell$ of a graph~$G$,
we let $ms_r(\ell)$ denote the largest integer~$s$
for which every $s$ consecutive edges of $\ell$ form a $(\leq r)$-regular subgraph of~$G$.
We define $ms_r(G)$ to be the maximum value of $ms_r(\ell)$ over all orderings $\ell$ of $G$.
In particular, the special case $ms_1(G)$, which we also denote as $ms(G)$,
is the same number as presented in the Introduction.
The edges $e_0 , \ldots , e_{s-1}$ of a graph $G = (V,E)$ are {\em cyclically consecutive} in $\ell$
if the labels of $e_0 , \ldots , e_{s-1}$ are consecutive integers modulo $|E|$.
We define $cms_r(\ell)$ and $cms_r(G)$ analogously to $ms_r(\ell)$ and $ms_r(G)$,
respectively, where we allow cyclically consecutive edges.
If $G$ is a $(\leq r)$-regular graph, then, by definition, $cms_r(G)= ms_r(G)=|E(G)|$.
So for the remainder of the paper, we only consider the more interesting case in which
$r$ is strictly less than the maximum degree of a vertex of $G$, denoted by $\Delta(G)$.
\begin{lemma}\label{lem: Gen matching sequence bound}
Let $G$ be a graph on $n$ vertices with $r<\Delta(G)$,
then
\[
  cms_r(G) \leq ms_r(G) \leq \left\lfloor\frac{rn-1}{2}\right\rfloor \,.
\]
\end{lemma}
\begin{proof}
If $rn$ is odd,
then a $(\leq r)$-regular graph on $n$ vertices can have at most $\frac{rn-1}{2}$ edges
and so $ms_r(G) \leq \left\lfloor\frac{rn-1}{2}\right\rfloor$.
If $rn$ is even,
then a $(\leq r)$-regular graph on $n$ vertices can have at most $\frac{rn}{2}$ edges.
If $\ell$ is a labelling of $G$ satisfying $ms_r(\ell) = \frac{rn}{2}$, then
the edges $\ell^{-1}(0), \ldots , \ell^{-1}(\frac{rn-2}{2})$ form a $r$-regular graph as do the edges
$\ell^{-1}(1), \ldots , \ell^{-1}(\frac{rn}{2})$. This means the edges
$\ell^{-1}(1), \ldots , \ell^{-1}(\frac{rn-2}{2})$, form a graph in which
every vertex has degree $r$ except two which have degree $r-1$. Therefore,
$\ell^{-1}(0) = \ell^{-1}(\frac{rn}{2})$, a contradiction.
Thus, $ms_r(G) \leq \left\lfloor\frac{rn-1}{2}\right\rfloor$.
The inequality $cms_r(G) \leq ms_r(G)$ is trivially true by definition.
\end{proof}
\noindent
For disjoint graphs $G_0 ,\ldots , G_{a-1}$
on the same vertex set $V$,
with labellings $\ell_{0} ,\ldots ,\ell_{a-1}$ respectively,
let $\ell_0 \vee\cdots\vee \ell_{a-1}$ denote
the ordering $\ell$ of $G= (V ,\bigcup_{i=0}^{a-1} E(G_i))$
defined by
$\ell(e_{ij}) = \ell_j (e_{ij}) + \sum_{l=0}^{j-1} |E(G_l)|$
where $e_{ij} \in E(G_j)$ for all $i$ and $j$.
Let $s$ be an integer and
$G$ and $G'$ be disjoint graphs on the same vertex set $V$
and each having at least $s-1$ edges.
Also, let $G$ and $G'$ have labellings $\ell$ and $\ell'$, respectively,
and let $G_s$ be the subgraph of $\left( V,E(G) \cup E(G')\right)$
that consists of the last $s-1$ edges of $\ell$ and the
first $s-1$ edges of $\ell'$.
Then we will let $\ell \vee_s \ell'$ denote the ordering of $G_s$ for which
the edges of $G_s$ appear in the same order as they do in $\ell \vee \ell'$.
Now we define $ms_r(\ell , \ell')$
to be the largest integer $s$ such that
$\ell \vee_{s} \ell'$
has $r$-matching sequencibility~$s$.

An {\em $r$-regular decomposition} of a graph $G$ is
a set of edge-disjoint $r$-regular subgraphs of $G$ that partition the edge set of~$G$.
A {\em $(\leq r)$-regular graph decomposition}
and a {\em matching decomposition} are defined analogously.

The main method used to prove Theorems
\ref{thm: Matching sequity for r}--\ref{thm: Gen cyclic matching sequence odd n}
is to decompose $K_n$ into regular parts
(regular in the sense that every vertex has the same degree),
then order the edges in each part,
and concatenate the parts to obtain an ordering for $K_n$.
The following propositions will facilitate this,
under certain conditions.
The propositions are given in more generality than we will require them,
as they may be useful for other matching sequencibility problems.
In each proposition,
the subscripts of the orderings $\ell_i$ are taken modulo $t$:
$\ell_{i+u} = \ell_{i'}$ exactly when $i' \equiv i+u \mod t$.

\begin{proposition}\label{prop: Matching decomposition both}
Let $G$ be a graph that decomposes into matchings $M_0 ,\ldots , M_{t-1}$, each with $n$ edges
and orderings $\ell_0 , \ldots , \ell_{t-1}$, respectively.
Suppose, for some $\epsilon \in [n]$ and $r < \Delta(G)$,
that $ms(\ell_i ,\ell_{i+r}) \geq n-\epsilon$ for all $i\in [t-r]$.
Then $ms_r(G) \geq rn-\epsilon$,
and if $ms(\ell_i ,\ell_{i+r}) \geq n-\epsilon$ for all $i\in [t]$,
then $cms_r(G) \geq rn-\epsilon$.
\end{proposition}

\begin{proposition}\label{prop: 2-regular decomposition both}
Let $r < \Delta(G)$ be even, set $u := \frac{r}{2}$,
and let $G$ be a graph that decomposes into $(\leq 2)$-regular graphs $R_0 ,\ldots , R_{t-1}$,
each with $n$ edges, and with orderings $\ell_0 , \ldots , \ell_{t-1}$,
respectively.
Suppose, for some non-zero $\epsilon \in \bigl[\lceil \frac{n}{2} \rceil \bigr] $,
that $ms_2 \left(\ell_i ,\ell_{i+u}\right) \geq n-\epsilon$ for all $i \in [t-u]$.
Then
$ms_r(G) \geq \frac{rn}{2}-\epsilon$,
and if $ms_2 \left(\ell_i ,\ell_{i+u}\right) \geq n-\epsilon$ for all $i\in [t]$,
then $cms_r(G) \geq \frac{rn}{2}-\epsilon$.
\end{proposition}

\begin{proposition}\label{prop: 2-regular decomposition odd r both}
Let $3 \leq r < \Delta(G)$ be odd, set $u := \frac{r-1}{2}$,
and let $G$ be a graph that decomposes into $(\leq 2)$-regular graphs $R_0 ,\ldots , R_{t-1}$,
each with $n$ edges, and with orderings $\ell_0 , \ldots , \ell_{t-1}$,
respectively.
Suppose, for some non-zero $\epsilon \in \bigl[\lceil \frac{n}{2} \rceil \bigr]$,
that $ms\left(\ell_{i} ,\ell_{i+u+1}\right) \geq \lceil \frac{n}{2}\rceil-\epsilon$ for all $i \in [t-u-1]$
and $ms_3\left(\ell_{i} ,\ell_{i+u}\right) \geq \lceil \frac{3n}{2} \rceil-\epsilon$ for all $ i \in[t-u]$.
Then $ms_r(G) \geq \lfloor \frac{rn+1}{2} \rfloor-\epsilon$,
and if $ms\left(\ell_{i} ,\ell_{i+u+1}\right) \geq \lceil\frac{n}{2} \rceil-\epsilon$
and $ms_3\left(\ell_{i} ,\ell_{i+u}\right) \geq \lceil \frac{3n}{2}\rceil-\epsilon$
for all $i \in [t]$,
then $cms_r(G) \geq \lfloor \frac{rn+1}{2} \rfloor-\epsilon$.
\end{proposition}

For a graph $G$ with ordering $\ell$,
$L_{\ell}(G)$ denotes the edges of $G$ listed the same order as~$\ell$ and say
that $\ell$ {\em corresponds} to $L_{\ell}(G)$;
i.e., if $e_0 ,\ldots, e_{k-1}$ is a list of the edges of $G$,
then $\ell$ corresponds to that list if $\ell(e_i) = i$ for all $i \in [k]$.
Also, for graphs
$G_0 , \ldots ,G_{a-1}$ with labellings $\ell_0 ,\ldots, \ell_{a-1}$, respectively,
$L_{\ell_0}(G_0) \vee\cdots\vee L_{\ell_{a-1}}(G_{a-1})$ denotes the lists of edges which the ordering
$\ell_0 \vee\cdots\vee \ell_{a-1}$ corresponds to.
The proofs of Proposition
\ref{prop: Matching decomposition both}--\ref{prop: 2-regular decomposition odd r both}
are very similar,
so we provide the proof of Proposition~\ref{prop: 2-regular decomposition odd r both}
and leave the details of the other two to the reader.
\begin{proof}[Proof of Proposition \ref{prop: 2-regular decomposition odd r both}]
The cyclic and non-cyclic cases are similar so we only show the cyclic case.
Let $\ell$ be the ordering corresponding to $L_{\ell_0}(R_{0}) \vee\cdots\vee L_{\ell_{t-1}}(R_{t-1})$.
Consider a set $E$ of $\lfloor \frac{rn+1}{2} \rfloor-\epsilon$ consecutive edges of $\ell$.
The edges of $E$, in order, will always be of the form
\[
\underbrace{e_1 ,\ldots, e_{j}}_\text{edges in $R_{i}$},
L_{\ell_{i+1}}\left(R_{i+1}\right) \vee\cdots\vee L_{\ell_{i+u+1-a}}\left(R_{i+u+1-a}\right),
\underbrace{e_{j+1} ,\ldots, e_{an-\lfloor\frac{n}{2}\rfloor-\epsilon}}_\text{edges in $R_{i+u+2-a}$} \,,
\]
for some $i \in[t]$, $j \in [n+1]$, and $a$.
Without loss of generality,
we can assume that
$E \cap E(R_{i})$ and
$E \cap E(R_{i+u+2-a})$ are non-empty and so $0 < j < an-\lfloor\frac{n}{2}\rfloor-\epsilon$.
There are $0 < an - \lfloor\frac{n}{2}\rfloor - \epsilon \leq 2n$ edges in
$E \cap \left(E(R_{i}) \cup E(R_{i+u+2-a}) \right)$.
Therefore, $a = 1$ or $a=2$.

If $a=1$, then the first $j$ edges and last $\lceil \frac{n}{2}\rceil-j-\epsilon$ edges of $E$
are the last $j$ edges of $\ell_{i}$ and the first $\lceil \frac{n}{2}\rceil-j-\epsilon$ edges of $\ell_{i+u+1}$,
respectively.
Thus,
the $j + \lceil \frac{n}{2}\rceil-j - \epsilon = \lceil \frac{n}{2} \rceil - \epsilon$
edges of $E \cap (E(R_i) \cup E(R_{i+u+1}))$ are consecutive in
$\ell_{i} \vee_{\lceil \frac{n}{2} \rceil - \epsilon} \ell_{i+u+1}$ and, by assumption,
form a matching.
The remaining edges of $E$ are from the $u = \frac{r-1}{2}$ $(\leq 2)$-regular graphs
$R_{i+1} , \ldots ,R_{i+u}$
and,
hence, the edges of $E$ form a $(\leq r)$-regular graph.

If $a=2$, then the first $j$ edges and last $\lceil \frac{3n}{2} \rceil -j-\epsilon$
edges of $E$ are the last $j$ edges of $\ell_{i}$
and the first $\lceil \frac{n}{2} \rceil-j-\epsilon$
edges of $\ell_{i+u}$, respectively.
Therefore, the $j + \lceil \frac{3n}{2} \rceil -j - \epsilon = \lceil \frac{3n}{2} \rceil - \epsilon$
edges of $E \cap (E(R_i) \cup E(R_{i+u}))$ are consecutive in
$\ell_{i} \vee_{\lceil \frac{3n}{2}\rceil - \epsilon} \ell_{i+u}$ and, by assumption, form a $(\leq 3)$-regular graph.
The remaining edges of $E$ are from the $u-1 = \frac{r-3}{2}$ $(\leq 2)$-regular graphs
$R_{i+1} , \ldots ,R_{i+u-1}$, and,
thus, the edges of $E$ form a $(\leq r)$-regular graph.
\end{proof}

\begin{remark}{\rm
Proposition \ref{prop: 2-regular decomposition odd r both}
holds for $r=1$ by replacing the assumption
$ms_3\left(\ell_{i} ,\ell_{i+u}\right) = ms_3\left(\ell_{i} ,\ell_{i}\right) \geq \lceil \frac{3n}{2}\rceil-\epsilon$ with $ms(\ell_i) \geq \lceil\frac{n}{2} \rceil - \epsilon $.}
\end{remark}
Note that Proposition \ref{prop: 2-regular decomposition odd r both}
requires two conditions on the orderings $\ell_0 , \ldots , \ell_{t-1}$, namely \\
$ms\left(\ell_{i}, \ell_{i+u+1}\right) \geq   \lceil \frac{n}{2} \rceil-\epsilon$
and
$ms_3\left( \ell_{i} ,\ell_{i+u}\right) \geq
\lceil \frac{3n}{2}\rceil-\epsilon$
(for the relevant $i$), whereas
Proposition~\ref{prop: Matching decomposition both}
and Proposition~\ref{prop: 2-regular decomposition both}
require only one condition.
This makes Proposition~\ref{prop: 2-regular decomposition odd r both}
harder to use and largely explains why Theorems
\ref{thm: Matching sequity for r}--\ref{thm: Gen cyclic matching sequence odd n}
cover more cases when $r$ and $n$ are not both odd.
Also, the requirement of two conditions in
Propositions \ref{prop: 2-regular decomposition odd r both}
may suggest that the case when $r$ and $n$ are odd is inherently
more difficult for $K_n$ (or even any graph of odd order).

Here and throughout the paper, we will allow orderings to be defined on sets of integers;
that is, a bijection
$\alpha \;:\; A \rightarrow [|A|]$
on a set $E$ of integers will also be considered an ordering.
However, we will only use such orderings
for re-indexing.
The following auxiliary lemma guarantees the existence of an ordering
of integers with particular useful properties.
It will find repeated use in later sections, so it is given here for easy reference.
\begin{lemma}\label{lem: General cyclic ordering}
Let $t$ and $u$ be integers with $t > u$ and set $d := \gcd(u,t)$.
Define $a_{i,j} := \left( i \pmod{\frac{t}{d}} \right)+j\frac{t}{d} \pmod{t}$
for all integers $i$ and $j$.
Then there exists an ordering $\alpha$ of $[t]$
with the property that
$\alpha(a_{i+1,j}) = \alpha(a_{i,j})+u \pmod{t}$
for all $i \in \bigl[\frac{t}{d}\bigr]$ and $j\in [d]$.
\end{lemma}

\begin{proof}
We check that the function
$\alpha \,:\, [t] \rightarrow [t]$
defined by
$\alpha(a_{i,j})= iu + j \pmod t$
for $i \in \bigl[\frac{t}{d}\bigr]$ and $j \in [d]$ will suffice.
First, we will show that $\alpha$ is a bijection.
Suppose that $iu+ j \equiv i' u + j'  \pmod t$,
with $i,i' \in\bigl[\frac{t}{d}\bigr]$ and $j,j' \in [d]$.
Then
$(i - i')u \equiv j'-j \pmod t$.
As $d$ divides $t$ and $u$, any multiple of $u$ modulo $t$ is also a multiple of $d$.
Thus,  $j-j'$ is a multiple of $d$, while $0 \leq |j-j'| \leq d-1$. This
is only possible if $j=j'$ and so $(i - i')u \equiv 0 \pmod t$.
As
$0 \leq |i' -i| \leq \frac{t}{d}-1$
and $\text{lcm}(t,u)= \frac{tu}{d}$,
we must also have that $i=i'$.
Thus, $\alpha$ is injective and so bijective; $\alpha$ is thus an ordering of $[t]$.
For any
$i \in \bigl[\frac{t}{d}\bigr]$ and $j \in [d]$,
\[
\alpha(a_{i+1,j}) = (i+1)u+j \pmod{t} =  iu+j +u \pmod{t} = \alpha(a_{i,j}) +u \pmod{t}\,.
\]
Hence, $\alpha$ has the required properties.
\end{proof}
\noindent
The function $\alpha$ in the proof can be used to show a
non-cyclic version of the lemma:
\begin{lemma}\label{lem: General non cyclic ordering}
Let $t > u$.
Then there exists an ordering $\alpha$ of $[t]$
with the property that, if $\alpha(a) \leq t-u-1$, then
$\alpha(a+1) = \alpha(a)+u$.
\end{lemma}

\begin{example}
{\rm Let $t = 10$ and $u = 4$;
then $d = \gcd(4,10) = 2$.
The following table summarises the function $\alpha$
that is produced by Lemma~\ref{lem: General cyclic ordering}.
\begin{center}
\begin{tabular}{|c|c|c|c|c|c|c|c|c|c|c|}
\hline
$x$          & $0$ & $1$ & $2$ & $3$ & $4$ & $5$ & $6$ & $7$ & $8$ & $9$ \\
\hline
$\alpha(x)$  & $0$ & $4$ & $8$ & $2$ & $6$ & $1$ & $5$ & $9$ & $3$ & $7$ \\
\hline
\end{tabular}
\end{center}}
\end{example}

\section{Decompositions of the complete graph \texorpdfstring{$K_n$}{Kn}}
\label{sec:decompositions-of-Kn}

To prove Theorems
\ref{thm: Matching sequity for r}--\ref{thm: Gen cyclic matching sequence odd n},
we will require
matching decompositions of $K_n$ when $n$ is even and
$2$-regular decompositions of $K_n$ when $n$ is odd,
so that we can apply the applicable
proposition from Section~\ref{Preliminaries}.
Here we present the required decompositions.

\subsection{Decompositions of \texorpdfstring{$K_n$}{Kn} for even \texorpdfstring{$n$}{n}}
\label{Decompositions of $K_n$ when $n$ is even}
Let $n = 2m$, $r \in [2m-1]-\{0\}$, $c \mid 2m-1$ and $d = \frac{2m-1}{c}$.
Let $V_{c,d} = \{v_\infty\} \cup \left\{ v_{i,j}\;:\;  i \in [c], j \in [d] \right\}$
be the vertex set of $K_{2m}$.
We set $v_{i,j} := v_{i',j'}$,
whenever $i' = i \pmod{c}$ and $j' = j \pmod{d}$.
The following sets (with the singleton excluded) are given in \cite{MR0389662}.
For an integer $x$ and odd integer~$y$,
let $P_{x,y} = \left\{ \left\{ x+l,x-l \right\}\;:\;  l\in  [\frac{y+1}{2}] \right\}$,
where the elements of the members of $P_{x,y}$ are taken modulo $y$.
We also note the following useful fact.

\begin{remark}\label{rem: the Ps are a partition}{\rm
Each family $P_{x,y}$ forms a partition of $[y]$ into pairs and a singleton,
and the set $\left\{ P_{x,y}\;:\; x \in [y]  \right\}$
partitions the pairs and singletons of $[y]$.}
\end{remark}

\noindent
For $i \in [c]$ and $j \in [d]$, let $M_{i,j}$
be the matching of $K_{2m}$ with edge set
\[
\bigl\{ \{ v_{\infty} ,v_{i,j} \} \bigr\} \cup
\bigl\{ \{ v_{a_1,b_1},v_{a_2,b_2} \} \;:\; \{ a_1,a_2 \} \in P_{i,c},
\{b_1,b_2 \} \in P_{j,d} \textrm{ and $a_1 \neq i$ or $b_1 \neq j$ } \bigr\} \,.
\]
These are indeed matchings, as the edge incident to $v_\infty$ in each $M_{i,j}$ is unique and
a vertex $v_{a_1,b_1}$, which is not adjacent to $v_\infty$ in $M_{i,j}$, is incident to
edges $\left\{ v_{a_1,b_1},v_{a_2,b_2} \right\}$ with $\{ a_1,a_2 \} \in P_{i,c}$
and $\{b_1,b_2 \} \in P_{j,d}$; also by the above remark,
the choice for such an $a_2$ and $b_2$ is unique.

\begin{example}
{\rm When $c=5$ and $d = 3$, $M_{0,0}$ is the following subgraph of $K_{2m}$.
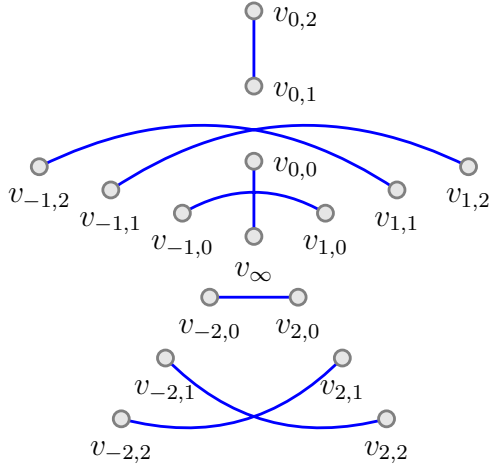
\begin{center}
\begin{tikzpicture}[thick,scale=.4]
  \pgfmathsetlengthmacro\scfac{2.5cm}
\foreach \angle/\colora/\mycount in {90/blue/1} %,162/red/2,234/green/3,%312/Yellow/4,18/violet/5
\draw[line width = \scfac*0.015, color = \colora]{
%% Adjacent to the centre vertex
%[bend right]
(0,0) -- ($(90+72*\intcalcMod{\mycount}{5}-72:1*\scfac)$)
%% Adjacent vertices of same radii
[bend right]
(18+72*\intcalcMod{\mycount}{5}  -72: 1*\scfac) to
(18+72*\intcalcMod{\mycount}{5}  +72: 1*\scfac)
(18+72*\intcalcMod{\mycount}{5} -144: 1*\scfac) --
(18+72*\intcalcMod{\mycount}{5} +144: 1*\scfac)
%
%% Adjacent vertices of same angle
(18+72*\intcalcMod{\mycount}{5} : 2*\scfac)--
(18+72*\intcalcMod{\mycount}{5} : 3*\scfac)
[bend right]
(18+72*\intcalcMod{\mycount}{5}-72 : 2*\scfac) to
(18+72*\intcalcMod{\mycount}{5}+72 : 3*\scfac)
[bend left]
(18+72*\intcalcMod{\mycount}{5}+72 : 2*\scfac) to
(18+72*\intcalcMod{\mycount}{5}-72 : 3*\scfac)
[bend left]
(18+72*\intcalcMod{\mycount}{5}-144 : 2*\scfac) to
(18+72*\intcalcMod{\mycount}{5}+144 : 3*\scfac)
[bend right]
(18+72*\intcalcMod{\mycount}{5}+144 : 2*\scfac) to
(18+72*\intcalcMod{\mycount}{5}-144 : 3*\scfac)
};
\foreach \r in {0,1,2}
\foreach \angle in {1,2,-1,-2}
\draw[line width = \scfac*0.015, color = black!50]{
%(\angle : \r)
(-\angle*72+90 : \r*\scfac+\scfac)
node [anchor=south,yshift=-\scfac*0.3,color = black]{$v_{\angle , \r}$}
node[circle, draw, fill=black!10,inner sep=\scfac*0.03, minimum width=\scfac*0.08]{
}
};
\draw[line width = \scfac*0.015, color = black!50]{
%(\angle : \r)
(-0*72+90 : 0*\scfac+\scfac)
node [anchor=west,yshift=-\scfac*0.04, xshift=\scfac*0.04, color = black]{$v_{0,0}$}
node[circle, draw, fill=black!10,inner sep=\scfac*0.03, minimum width=\scfac*0.08]{
}
(-0*72+90 : 1*\scfac+\scfac)
node  [anchor=west,yshift=-\scfac*0.04, xshift=\scfac*0.04, color = black]{$v_{0,1}$}
node[circle, draw, fill=black!10,inner sep=\scfac*0.03, minimum width=\scfac*0.08]{
}
(-0*72+90 : 2*\scfac+\scfac)
node  [anchor=west,yshift=-\scfac*0.04, xshift=\scfac*0.04, color = black]{$v_{0,2}$}
node[circle, draw, fill=black!10,inner sep=\scfac*0.03, minimum width=\scfac*0.08]{
}
};
\draw[line width = \scfac*0.015, color = black!50]{
(0 : 0)
node [anchor=south,yshift=-\scfac*0.3,color = black]{$v_{\infty}$}
node [circle, draw, fill=black!10,inner sep=\scfac*0.03, minimum width=\scfac*0.08]{}
};
\end{tikzpicture}
\end{center}}
\end{example}

\begin{lemma}
The set $\{ M_{i,j} \;:\; i\in [c],\, j \in [d]\}$
is a matching decomposition of $K_{n}$.
\end{lemma}
\begin{proof}
Clearly, $M_{i,j}$ is the unique matching containing the edge
$\{v_{\infty}, v_{i,j} \}$.
For $a_1,a_2 \in [c]$ and $b_1,b_2 \in [d]$, with $a_1 \neq a_2$ or $b_1 \neq b_2$,
the edge $\{v_{a_1,b_1},v_{a_2,b_2} \}$ is in the matching $M_{i,j}$,
for some $i \in [c]$ and $j \in [d]$ if
$\{a_1,a_2 \} \in P_{i,c}$ and $\{b_1,b_2 \} \in P_{j,d}$.
By Remark~\ref{rem: the Ps are a partition},
such $i$ and $j$, and therefore $M_{i,j}$ exist and are uniquely given.
Thus, the edge $\{v_{a_1,b_1},v_{a_2,b_2} \}$ occurs in
exactly one matching.
Hence, $\{ M_{i,j} \;:\; i \in [c], j\in [d] \}$
is a matching decomposition of~$K_{n}$.
\end{proof}

Note that this decomposition is the same for different values of $c$, just indexed differently.
Indeed,
the bijection $\tau_{c,d} : V_{c,d} \rightarrow V_{2m-1,1}$
defined by $\tau_{c,d}(v_{\infty}) = v_{\infty}$ and
$\tau_{c,d} (v_{a,b}) = v_{ad+b,0}$ for $a \in [c]$ and $b \in [d]$
is an isomorphism,
showing that the decomposition for a particular value of $c$
is isomorphic to the decomposition for $c = 2m-1$.
Note that the Walecki decomposition (see \cite{MR2394738}) decomposes
$K_{2m}$ into Hamiltonian cycles and a complete matching,
from which the matching decomposition for $c = 2m-1$,
given above, can be easily obtained.

\subsection{Decompositions of \texorpdfstring{$K_n$}{Kn} when \texorpdfstring{$n$}{n} is odd}
\label{Decompositions of $K_n$ when $n$ is odd}

Let $n = 2m+1$.
We will present two different decompositions for $K_{2m+1}$.
The first is the Walecki decomposition~\cite{MR2394738} mentioned above.
Let $V = \{ \infty \} \cup \mathbb{Z}_{2m}$
be the vertex set of~$K_{2m+1}$.
Let $H_0$ be the Hamiltonian cycle
$\infty \,, 0  \,, 1 \,,  -1 \,, 2 \,, -2 \,, \ldots \,, x \,, -x \,, \ldots \,, m-1 , -(m-1) , m$,
as depicted in Figure \ref{fig1}.
\begin{figure}
\begin{center}
\begin{tikzpicture}[scale = 0.7]
\pgfmathsetlengthmacro\n{360/10}
\pgfmathsetlengthmacro\scfac{2.5cm}
\pgfmathsetlengthmacro\r{\scfac*1.5}
\draw[line width = \scfac*0.015, dashed, color =blue]{(162:\r) --($0.5*(162:\r)+0.5*(-18:\r)$)};
\draw[line width = \scfac*0.015, dashed, color =blue]{(-54:\r) --($0.5*(198:\r)+0.5*(-54:\r)$)};
\draw[line width = \scfac*0.01 ,color = black]{
(189:\r) node[circle, draw, fill,inner sep=\scfac*0.015, minimum width=\scfac*0.01]{}
(198:\r) node[circle, draw, fill,inner sep=\scfac*0.015, minimum width=\scfac*0.01]{}
(207:\r) node[circle, draw, fill,inner sep=\scfac*0.015, minimum width=\scfac*0.01]{}
( -9:\r) node[circle, draw, fill,inner sep=\scfac*0.015, minimum width=\scfac*0.01]{}
(-18:\r) node[circle, draw, fill,inner sep=\scfac*0.015, minimum width=\scfac*0.01]{}
(-27:\r) node[circle, draw, fill,inner sep=\scfac*0.015, minimum width=\scfac*0.01]{}
};
\draw[line width = \scfac*0.015, color = blue]{
(-90 :\r  ) -- (90:\r/6)
( 90 :\r/6) -- (90:\r)
};
\draw[line width = \scfac*0.015, color = blue]
\foreach \x/\y/\z in {90/54/0,54/126/1,126/18/-1,18/162/2,-54/234/m-1,234/-90/-(m-1)}
{
 (\x :\r)-- (\y:\r)
};

\draw[line width = \scfac*0.015, color = black!50] \foreach \x/\y/\z in {90/54/0,54/126/1,126/18/-1,18/162/2}
{
 (\x :\r) node [anchor=north,yshift=\scfac*0.3,color = black]{$\z$}  node[circle, draw, fill=black!10,inner sep=\scfac*0.03, minimum width=\scfac*0.08]{}
};
\draw[line width = \scfac*0.015, color = black!50] \foreach \x/\y/\z in {-54/234/m-1,234/-90/-(m-1)}
{
 (\x :\r)node[circle, draw, fill=black!10,inner sep=\scfac*0.03, minimum width=\scfac*0.08]{}
};
\draw[line width = \scfac*0.015, color = black!50] \foreach \x/\y/\z in {-54/234/m-1,234/-90/-(m-1)}
{
 (\x :\r*1.05) node [anchor=south,yshift=-\scfac*0.31,color = black]{$\z$}
};
\draw[line width = \scfac*0.015, color = black!50]{
(162:\r) node [anchor=north,yshift=\scfac*0.31,color = black]{$-2$}
node[circle, draw, fill=black!10,inner sep=\scfac*0.03, minimum width=\scfac*0.08]{}
};
\draw[line width = \scfac*0.015, color = black!50]{
(-90 :\r)node [anchor=south,yshift=-\scfac*0.3,color = black]{$m$}
node[circle, draw, fill=black!10,inner sep=\scfac*0.03, minimum width=\scfac*0.08]{}
(90 :\r/6)node [anchor=west,xshift=\scfac*0.07,yshift=-\scfac*0.02,color = black]{$\infty$}
node[circle, draw, fill=black!10,inner sep=\scfac*0.03, minimum width=\scfac*0.08]{}
};
\end{tikzpicture}
\end{center}
\caption{The Hamiltonian cycle $H_0$}\label{fig1}
\end{figure}
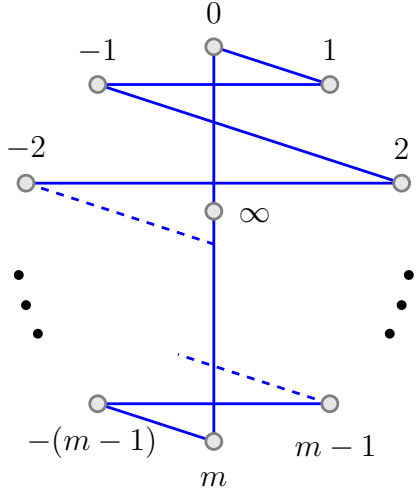

\noindent
Let $\sigma$ be the permutation $\sigma = (\infty)(0 \;\; 1 \;\;\cdots\;\;  2m-2 \;\; 2m-1)$.
Let $H_i = \sigma^{i}(H_0)$ for $i \in [m]$,
where $\sigma$ acts on the vertices of $V$.
Alspach~\cite{MR2394738} proved the following lemma,
and we give a similar proof for completeness.
\begin{lemma}
The set $\left\{ H_0,\ldots ,H_{m-1} \right\}$ is
a $2$-regular decomposition of $K_{2m+1}$.
\end{lemma}
\begin{proof}
As each $H_i$ has $2m+1$ edges, we only
need to show that the edges of $H_0 , \ldots , H_{m-1}$
are disjoint.
Clearly the edges $\{ \infty,i\}$ and $\{ \infty, i+m\}$
are only present in $H_i$, for all $i \in [m]$.
For the remaining edges,
let the {\em length} of an edge $\{i,j\}$ with $i,j \neq \infty$ be $j-i$ mod $2m$ or $i-j$ mod $2m$,
whichever lies in $[m+1]- \{0\}$.

We check, for every fixed length $l \in [m+1]-\{0\}$, that the edges of length $l$ in $H_0,\ldots, H_{m-1}$ are distinct.
Note that the edges of $H_i$ that are not incident to $\infty$ are
$\{ i+x,i-x\}$ for $x \in [m]-\{0\}$,
and $\{ i+x,i-x+1 \}$ for $x \in [m+1]-\{0\}$.
The edges of even length $l < m$ in $H_i$  are
$\left\{i+\frac{l}{2} , i-\frac{l}{2} \right\}$ and
$ \left\{i+m -\frac{l}{2} , i-m+\frac{l}{2} \right\}$,
and neither edge is an edge of $H_j$ for $j \neq i$.
The edges of odd length $l < m$ in $H_i$ are
$\left\{ i+\frac{l+1}{2} , i-\frac{l+1}{2}+1 \right\}$ and
$\left\{ i+m-\frac{l-1}{2} , i-m+\frac{l-1}{2}+1 \right\}$,
and neither edge is an edge of $H_j$ for $j \neq i$.
If $m$ is even, then the edge of length $m$
in $H_i$
is $\left\{ i+\frac{m}{2} , i-\frac{m}{2} \right\}$,
and is only an edge of $H_i$.
If $m$ is odd, then the edge of length $m$
in $H_i$ is
$ \left\{ i+\frac{m+1}{2} , i-\frac{m+1}{2}+1 \right\}$,
and is only an edge of $H_i$.
Therefore, the
edges of every length $l \in [m+1]-\{0\}$ in $H_0 , \ldots , H_{m-1}$ are distinct.
\end{proof}

The $2$-regular decomposition $\left\{ H_0,\ldots ,H_{m-1} \right\}$
of $K_{2m+1}$ has a particular disadvantage relevant to us.
If $\ell_0$ is an ordering of $H_0$, then the permutation $\sigma$
induces orderings $\ell_1 , \ldots , \ell_{m-1} , \ell_{m}$
of $H_1 , \ldots , H_{m-1}, H_{m} = H_0$, respectively.
However, $\ell_{m} \neq \ell_0$.
Therefore, the $2$-regular decomposition $\{H_0,\ldots ,H_{m-1}\}$
is not ideal for constructing cyclic orderings in this fashion.

The second decomposition overcomes this problem
but does not exist for all $n$.
Recall that $n = 2m+1$ and let $m$ be odd.
Let $V_{m,2} = \{v_{\infty} \} \cup \{v_{i,j}\;:\;
i \in [m], j \in [2] \}$
be the vertex set of $K_{2m+1}$.
For an integer $x$ and an odd integer $y$, let $P_{x,y}$
be as defined in Subsection~\ref{Decompositions of $K_n$ when $n$ is even}.
Let $R_{i}$ be the subgraph of $K_{2m+1}$ with edge set
\[
\bigl\{ \{ v_{\infty},v_{i,0} \},\{ v_{\infty},v_{i,1}\}, \{v_{i,0},v_{i,1}\} \bigr\} \cup
\bigl\{ \{ v_{a_1,b_1},v_{a_2,b_2} \} \;:\;
\{ a_1,a_2 \} \in P_{i,m} , a_1 \neq i,  b_1,b_2 \in [2]\bigr\} \,.
\]

\newpage
\begin{example}
{\rm The following graph depicts $R_0,R_1$ and $R_2$ for $K_7$.
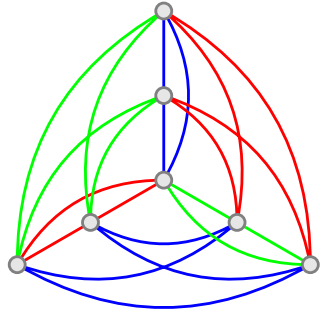
\begin{center}
\begin{tikzpicture}[thick,scale=0.45]
  \pgfmathsetlengthmacro\scfac{2.5cm}
  \pgfmathsetmacro{\sepang}{120}
  \pgfmathtruncatemacro{\c}{3}
\foreach \colora/\mycount in {blue/1, red/2,green/3}
\draw[line width = \scfac*0.015, color = \colora]{
%% Adjacent to the centre vertex
%[bend right]
(0,0)--($(90+\sepang*\intcalcMod{\mycount-1}{\c}:1*\scfac)$)
[bend right]
(0,0)to ($(90+\sepang*\intcalcMod{\mycount-1}{\c}:2*\scfac)$)
%% Adjacent vertices of same radii
[bend right]
(90+\sepang*\intcalcMod{\mycount-1}{\c}+\sepang: 2*\scfac) to
(90+\sepang*\intcalcMod{\mycount-1}{\c}-\sepang: 2*\scfac)
[bend right]
(90+\sepang*\intcalcMod{\mycount-1}{\c}+\sepang: 1*\scfac) to
(90+\sepang*\intcalcMod{\mycount-1}{\c}-\sepang: 1*\scfac)
%
%% Adjacent vertices of same angle
(90+\sepang*\intcalcMod{\mycount-1}{\c} : 1*\scfac)--
(90+\sepang*\intcalcMod{\mycount-1}{\c} : 2*\scfac)
%
%%% Remaining edges
[bend right]
(90+\sepang*\intcalcMod{\mycount-1}{\c}+\sepang : 1*\scfac)to
(90+\sepang*\intcalcMod{\mycount-1}{\c}-\sepang : 2*\scfac)
[bend left]
(90+\sepang*\intcalcMod{\mycount-1}{\c}-\sepang : 1*\scfac)to
(90+\sepang*\intcalcMod{\mycount-1}{\c}+\sepang : 2*\scfac)
};
\pgfmathsetmacro{\sepangint}{90+\sepang}
\pgfmathsetmacro{\sepangfin}{450-\sepang}
\foreach \r in {1*\scfac,2*\scfac}
\foreach \angle in {90,\sepangint,...,\sepangfin}
\draw[line width = \scfac*0.015, color = black!50]{
(\angle: \r) node[circle, draw, fill=black!10,inner sep=\scfac*0.03, minimum width=\scfac*0.08]{}
};
\draw[line width = \scfac*0.015, color = black!50]{
(0 : 0) node[circle, draw, fill=black!10,inner sep=\scfac*0.03, minimum width=\scfac*0.08]{}
};

\end{tikzpicture}
\end{center}}
\end{example}
Clearly $R_0,R_1$ and $R_2$ form a $2$-regular decomposition of $K_7$.
Now we will show that the same is true in general.

\begin{lemma}
For odd $m$, %the set
$\left\{ R_0 , \ldots , R_{m-1} \right\}$ is a $2$-regular
graph decomposition of $K_{2m+1}$.
\end{lemma}
\begin{proof}
Clearly, the edges
$\left\{ v_{\infty},v_{i,0} \right\},\left\{ v_{\infty},v_{i,1} \right\}$
and $\left\{ v_{i,0},v_{i,1} \right\}$
are present only in $R_i$.
The edge $\{ v_{a_1,b_1},v_{a_2,b_2} \}$ with $a_1 \neq a_2$ is present in $R_i$
for $i$ such that $\{a_1,a_2 \} \in P_{i,m}$.
By
Remark \ref{rem: the Ps are a partition},
such an
$i$ exists and is unique. Thus, every edge is in a unique $R_i$
for $i \in [m]$.
\end{proof}

\section{Proof of Theorems \ref{thm: Matching sequity for r} and \ref{thm: Cyclic matching sequity for r} for when \texorpdfstring{$n$}{n} is even}
\label{sec: Gen matching sequence}
Write $n = 2m$ and let $r \in [n-1] -\{0\}$.
Set $d := \gcd(r,2m-1)$ and $c := \frac{2m-1}{d}$,
and as in Subsection~\ref{Decompositions of $K_n$ when $n$ is even},
define $V_{c,d}$ to be the vertex set of $K_{2m}$.
Also, let $M_{i,j}$ be the matchings
defined in Subsection~\ref{Decompositions of $K_n$ when $n$ is even} for $i \in [c]$ and $j \in [d]$.
Define $\ell_{i,j}$ to be the following ordering of $M_{i,j}$:
\begin{align*}
\ell_{i,j}(\{v_{\infty},v_{i,j} \}) &=0\,, \\
\ell_{i,j}( \{ v_{i+x,j}, v_{i-x,j} \}) &= x
\quad\text{for}\quad x \in \left[ \dfrac{c+1}{2} \right]
-\{ 0\}\,, \\
\begin{split}
\ell_{i,j}( \{ v_{i+2x,j+y}, v_{i-2x,j-y} \}) &=
(y-1)c+\frac{c+1}{2}+\Bigl(x+i\frac{c-1}{2} \hspace*{-3mm}\pmod{c}\Bigr)\\
&\;\text{ $ $ } \quad\quad\text{for}\quad x \in [c],\, y \in \left[ \dfrac{d+1}{2} \right]
-\{ 0\}\,.
\end{split}
\end{align*}

\begin{example}
{\rm When $c=5$ and $d=3$, the matching $M_{0,0}$ is labelled as in Figure \ref{fig2}.}

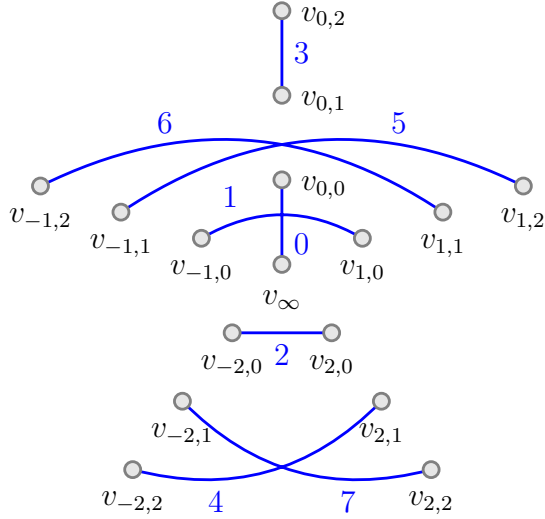
\begin{figure}\begin{center}
\begin{tikzpicture}[thick,scale=.45]
  \pgfmathsetlengthmacro\scfac{2.5cm}
\foreach \angle/\colora/\mycount in {90/blue/1} %,162/red/2,234/green/3,
\draw[line width = \scfac*0.015, color = \colora]{
%% Adjacent to the centre vertex
%[bend right]
(0,0) -- node[pos = 0.5, below right] {0}
($(90+72*\intcalcMod{\mycount}{5}-72:1*\scfac)$)
%% Adjacent vertices of same radii
[bend right] (18+72*\intcalcMod{\mycount}{5}-72: 1*\scfac) to node[pos = 0.7, above left] {  1}
(18+72*\intcalcMod{\mycount}{5} +72: 1*\scfac)
(18+72*\intcalcMod{\mycount}{5}-144: 1*\scfac) -- node[pos = 0.5, below ] {2}
(18+72*\intcalcMod{\mycount}{5} +144: 1*\scfac)
(18+72*\intcalcMod{\mycount}{5} : 2*\scfac)-- node[pos = 0.5, right ] {3}
(18+72*\intcalcMod{\mycount}{5} : 3*\scfac)
%%%%%
[bend right]
(18+72*\intcalcMod{\mycount}{5}-72 : 2*\scfac) to node[pos = 0.7, above ] {6}
(18+72*\intcalcMod{\mycount}{5}+72 : 3*\scfac)
[bend left]
(18+72*\intcalcMod{\mycount}{5}+72 : 2*\scfac) to node[pos = 0.7, above ] {5}
(18+72*\intcalcMod{\mycount}{5}-72 : 3*\scfac)
[bend left]
(18+72*\intcalcMod{\mycount}{5}-144 : 2*\scfac) to node[pos = 0.7, below ] {4}
(18+72*\intcalcMod{\mycount}{5}+144 : 3*\scfac)
[bend right]
(18+72*\intcalcMod{\mycount}{5}+144 : 2*\scfac) to node[pos = 0.7, below ] {7}
(18+72*\intcalcMod{\mycount}{5}-144 : 3*\scfac)
};
\foreach \r in {0,1,2}
\foreach \angle in {1,2,-1,-2}
\draw[line width = \scfac*0.015, color = black!50]{
(-\angle*72+90 : \r*\scfac+\scfac)
node [anchor=south,yshift=-\scfac*0.3,color = black]{$v_{\angle,\r}$}
node[circle, draw, fill=black!10,inner sep=\scfac*0.03, minimum width=\scfac*0.08]{}};
\draw[line width = \scfac*0.015, color = black!50]{
%(\angle : \r)
(-0*72+90 : 0*\scfac+\scfac)
node  [anchor=west,yshift=-\scfac*0.04, xshift=\scfac*0.04, color = black]{$v_{0,0}$}
node[circle, draw, fill=black!10,inner sep=\scfac*0.03, minimum width=\scfac*0.08]{
}
(-0*72+90 : 1*\scfac+\scfac)
node  [anchor=west,yshift=-\scfac*0.04, xshift=\scfac*0.04, color = black]{$v_{0,1}$}
node[circle, draw, fill=black!10,inner sep=\scfac*0.03, minimum width=\scfac*0.08]{
}
(-0*72+90 : 2*\scfac+\scfac)
node  [anchor=west,yshift=-\scfac*0.04, xshift=\scfac*0.04, color = black]{$v_{0,2}$}
node[circle, draw, fill=black!10,inner sep=\scfac*0.03, minimum width=\scfac*0.08]{}};
\draw[line width = \scfac*0.015, color = black!50]{
(0 : 0)
node [anchor=south,yshift=-\scfac*0.3,color = black]{$v_{\infty}$}
node[circle, draw, fill=black!10,inner sep=\scfac*0.03, minimum width=\scfac*0.08]{}
};
\end{tikzpicture}
\end{center}
\caption{The matching $M_{0,0}$}\label{fig2}
\end{figure}
\end{example}

\noindent
Note that the matchings $M_{i,0}$ are obtained by rotating the
above graph but the orderings $\ell_{i,0}$ are not.
We will use the following notation.
Let $V'_{j,0} = \{ v_{\infty}\} \cup \{v_{z,j}\;: \; z \in [c]\}$ for $j \in [d]$
and let $V'_{j,y} = \{v_{z,j \pm y}\;: \; z \in [c]\}$ for $j \in [d]$ and
$y \in \left[ \frac{d+1}{2} \right]-\{0\}$.
Clearly, $V'_{j,0} , \ldots , V'_{j,\frac{d-1}{2}}$
partition $V_{c,d}$ for all $j \in [d]$.
The crucial component of the proof of
Theorems~\ref{thm: Matching sequity for r} and~\ref{thm: Cyclic matching sequity for r}
for when $n$ is even
is the following lemma which allows us to apply
Proposition~\ref{prop: Matching decomposition both}.

\begin{lemma}\label{lem: Gen matching seq even n}
For all $i \in [c]$ and $j \in [d]$,
$ms(\ell_{i,j},\ell_{i+1,j}) \geq m-1$.
\end{lemma}

\begin{proof}
Consider a set of $m-1$ consecutive edges $E$ in
$\ell = \ell_{i,j} \vee_{m-1} \ell_{i+1,j}$.
As $M_{i,j}$ and $M_{i+1,j}$ are matchings,
two edges incident to a common vertex in $E$ cannot both be from
$M_{i+1,j}$ or both be from $M_{i,j}$.
The edges of $E$ are the edges labelled $m-l , \ldots , m-1$
by $\ell_{i,j}$ and the edges labelled $0  , \ldots , m-l-2$
by $\ell_{i+1,j}$,
for some $l \in [m-1] -\{0\}$.
Thus, for a vertex $v$ to be incident to
two edges in $E$,
$v$ must be incident to
$e_1$ in $M_{i,j}$ and $e_2$ in $M_{i+1,j}$
which satisfy  $\ell_{i,j}(e_1) \geq m-l$
and $\ell_{i+1,j}(e_2) \leq m-l-2$.
Therefore, $\ell_{i,j}(e_1)- \ell_{i+1,j}(e_2) \geq m-l-(m-l-2) = 2$.
Hence, to check that $E$ forms a matching,
it suffices to show,
for all vertices $v \in V_{c,d}$, that
if $v$ is incident to $e_1$ in $M_{i,j}$ and $e_2$ in $M_{i+1,j}$,
then $\ell_{i,j}(e_1)- \ell_{i+1,j}(e_2) < 2$.
Let $v \in V'_{j,y}$ for some $ y \in [ \frac{d+1}{2} ]$.

First, suppose that $y=0$.
If $e_1$ in $M_{i,j}$ and $e_2$ in $M_{i+1,j}$
are both incident to $\infty$,
then $\ell_{i,j}(e_1)-\ell_{i+1,j}(e_2) = 0-0 = 0 <  2$.
We therefore only need to check the remaining vertices in $V'_{j,0}$.
Let $v = v_{i + x,j}$ for some $x \in [ \frac{c+1}{2} ]$.
Then $v$ is incident to the edge labelled $x$ by $\ell_{i,j}$.
Let $x' \in [ \frac{c+1}{2} ]$ be the integer such that
$v$ is either $v_{i+1 + x',j}$ or $v_{i+1 - x',j}$.
In either case, $v$ is incident to the edge labelled $x'$ by $\ell_{i+1,j}$.
Therefore, it
suffices to show that $x-x' < 2$.
If $v = v_{i+1 + x',j}$, then $i+ x \equiv i+1+ x' \pmod{c}$
and so $x \equiv x' +1 \pmod{c}$.
As $x,x' \in [ \frac{c+1}{2} ] $,
it follows that $x = x' + 1$,
and so $x-x' = 1 < 2$.
Otherwise, $v = v_{i+1 - x',j}$,
implying that $i+ x \equiv i+1- x' \pmod{c}$,
and so $x +x' \equiv 1 \pmod{c}$.
As $ x ,x'\in [\frac{c+1}{2}]$,
it follows that $\{x,x' \} = \{ 0,1\}$.
Therefore, $x-x'$ is $1$ or $-1$ and in particular less than $2$.
The case in which $v = v_{i-x,j}$ for some $ x \in [\frac{c+1}{2}]$
can be treated in a similar fashion and is left to the reader.

Now suppose that $y \neq 0$.
Let
$v =  v_{i + 2x , j + y}$
for some $x \in [c]$.
Let $e_1$ and $e_2$ be the edges incident to $v$ in $M_{i,j}$ and $M_{i+1,j}$, respectively.
Let $v = v_{i+1 + 2x' , j+ y}$ for some $x'\in [c]$.
Then $i + 2x \equiv i+1 + 2x' \pmod{c}$.
As $\gcd(2,c) = 1$, this reduces to
\begin{equation}\label{eqn:cond on vertices 1}
 x' \equiv x + \frac{c -1 }{2}   \pmod{c}\,.
\end{equation}

We will now check that $\ell_{i,j}(e_1) - \ell_{i+1,j}(e_2) < 2$.
The labels of $e_1$ and $e_2$ are, respectively,
\[
  (y-1)c+\frac{c+1}{2}+\Bigl( x + i   \frac{c-1}{2} \!\pmod{c}\Bigr) \;\;\text{and}\;\;
  (y-1)c+\frac{c+1}{2}+\Bigl( x'+(i+1)\frac{c-1}{2} \!\pmod{c}\Bigr)\,.
\]
By (\ref{eqn:cond on vertices 1}),
the difference between these two labels is
\[
\Bigl( x+i\frac{c-1}{2} \pmod{c}\Bigr) -
\Bigl( x+i\frac{c-1}{2} +(1 + 1)\frac{c-1}{2}\pmod{c}\Bigr)\,.
\]
The right term is $c-1$ more than the left term before taking modulo $c$.
So the difference is either $1$ or $-(c-1)$ and in particular less than $2$.
Hence, $\ell_{i,j}(e_1) - \ell_{i+1,j}(e_2) < 2$.
The case in which $v = v_{i - 2x,j-y}$ is similar and therefore omitted.
Thus, $E$ forms a matching.
\end{proof}

\begin{proof}
[Proof of Theorems \ref{thm: Matching sequity for r} and
\ref{thm: Cyclic matching sequity for r} when $n$ is even]
Theorem 4 when $n$ is even follows from Theorem 5 when $n$ is even, so we only
prove the latter.
Let $\alpha$ and $a_{i,j}$ be as defined in Lemma \ref{lem: General cyclic ordering} for $u = r$ and $t = 2m-1$.
Let $M'_{\alpha(a_{i,j})} = M_{i,j}$ and $\ell_{\alpha(a_{i,j})} = \ell_{i,j}$ for  all $i \in [c]$ and $j \in [d]$.
If $x = \alpha(a_{i,j})$,
then, by Lemmas~\ref{lem: General cyclic ordering} and~\ref{lem: Gen matching seq even n},
$ms(\ell_{x},\ell_{x+r}) = ms(\ell_{i,j},\ell_{i+1,j}) \geq m-1$
(where $x+r$ in $\ell_{x+r}$ is taken modulo $2m-1$).
Thus,
Proposition \ref{prop: Matching decomposition both} yields $cms_r(K_{2m})\geq rm-1$,
using the matchings $M'_{0} ,\ldots ,M'_{2m-1}$ ordered by
$\ell'_{0} ,\ldots ,\ell'_{2m-1}$, respectively.
The reverse inequality, $cms_r(K_{2m})\leq rm-1$, follows from
Lemma~\ref{lem: Gen matching sequence bound}.
This completes the proof.
\end{proof}

\section{Proof of Theorem \ref{thm: Matching sequity for r} for when \texorpdfstring{$n$}{n} is odd and \texorpdfstring{$\gcd (r,n-1) = 1$}{gcd (r,n-1) = 1}}
Let $n = 2m+1$ and $r \in [2m]-\{0\}$ be an integer such that $\gcd(r,2m) = 1$.
Also, let $V_{2m}$ be the vertex set of $K_{2m+1}$, and
$H_i$ be the Hamiltonian cycles
from Subsection~\ref{Decompositions of $K_n$ when $n$ is odd}
for $i \in [m]$.
Let $\ell_i$ be the ordering of $H_i$ defined as follows:
\begin{align*}
\ell_{i}(\{ \infty, i    \}) &= 0\,, \\
\ell_{i}(\{ \infty, i+m  \}) &= m\,, \\
\ell_{i}(\{  i+rx , i-rx \}) &= x \quad\text{for non-zero}\; x \in [m]\,, \\
\ell_{i}\Bigl(\Bigl\{ i+rx+\frac{r+1}{2}, i-rx-\frac{r-1}{2} \Bigr\}\Bigr) &= m+x+1 \quad\textrm{for}\quad x \in [m]\,.
\end{align*}

This is indeed valid as
the edge $\{i+a,i-a \}$ has label $ar^{-1} \pmod{m}$ for $a \in [m] - \{0\}$
and the edge $\{ i+a, i-a+1 \}$
has the label $m+1+\left(r^{-1}\left( a -\frac{r+1}{2}\right) \pmod{m}\right)$ for $a \in [m]$.
Also,  it is clear that the first $m$ edges of $\ell_i$ form a matching
as do the last $m$ edges of $\ell_i$.

\begin{example}
{\rm
When $n = 11$ and $r= 3$, the Hamiltonian cycle $H_0$ is labelled as in Figure~\ref{fig3}.}
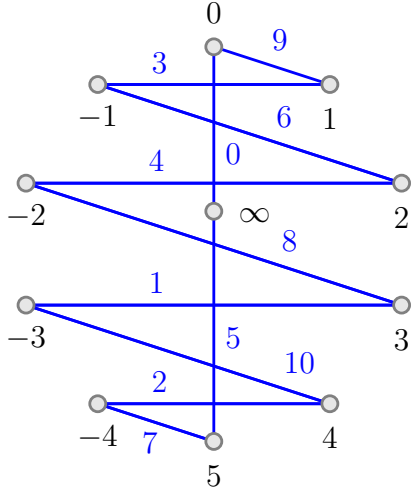
\begin{figure}
\begin{center}
\begin{tikzpicture}[scale = .7]
\pgfmathsetlengthmacro\n{360/10}
\pgfmathsetlengthmacro\scfac{2.5cm}
\pgfmathsetlengthmacro\r{\scfac*1.5}
\draw[line width = \scfac*0.015, color = blue]{
(-90 :\r)   -- node[pos = 0.45, right ] {5} ( 90:\r/6)
( 90 :\r/6) -- node[pos = 0.35, right ] {0} ( 90:\r)
( 90 :\r)   -- node[pos = 0.4, above right ] {9} (54:\r)
};
\draw[line width = \scfac*0.015, color = blue]
{
 (90+1*36:\r)-- node[pos = 0.35, above left  ] {3} (90-1*36:\r)
 (90+2*36:\r)-- node[pos = 0.4 , above left  ] {4} (90-2*36:\r)
 (90+3*36:\r)-- node[pos = 0.4 , above left  ] {1} (90-3*36:\r)
 (90+4*36:\r)-- node[pos = 0.35, above left  ] {2} (90-4*36:\r)
 (90+1*36:\r)-- node[pos = 0.55, above right ] {6} (90-1*36-36:\r)
 (90+2*36:\r)-- node[pos = 0.65, above right ] {8} (90-2*36-36:\r)
 (90+3*36:\r)-- node[pos = 0.81, above right ] {10} (90-3*36-36:\r)
 (90+4*36:\r)-- node[pos = 0.45, below       ] {7} (90-4*36-36:\r)
};

\draw[line width = \scfac*0.015, color = blue] \foreach \angle in {1,2,3,4}
{
 (90+\angle*36:\r)--
 (90-\angle*36:\r)
 (90+\angle*36:\r)--
 (90-\angle*36-36:\r)
};

\draw[line width = \scfac*0.015, color = black!50] \foreach \angle in {1,2,3,4}
{
 (90+\angle*36:\r) node [anchor=south,yshift=-\scfac*0.3, color = black]{$-\angle$} node[circle, draw, fill=black!10,inner sep=\scfac*0.03, minimum width=\scfac*0.08]{}
 (90-\angle*36:\r) node [anchor=south,yshift=-\scfac*0.3, color = black]{$ \angle$} node[circle, draw, fill=black!10,inner sep=\scfac*0.03, minimum width=\scfac*0.08]{}
};
\draw[line width = \scfac*0.015, color = black!50]{
(-90 :\r) node [anchor=south,yshift=-\scfac*0.3,color = black]{$5$}
node[circle, draw, fill=black!10,inner sep=\scfac*0.03, minimum width=\scfac*0.08]{}
(90 :\r/6) node [anchor=west,xshift=\scfac*0.07,yshift=-\scfac*0.02,color = black]{$\infty$}
node[circle, draw, fill=black!10,inner sep=\scfac*0.03, minimum width=\scfac*0.08]{}
(90 :\r) node [anchor=north,yshift=\scfac*0.3,color = black]{$0$}
node[circle, draw, fill=black!10,inner sep=\scfac*0.03, minimum width=\scfac*0.08]{}
};
\end{tikzpicture}
\end{center}
\caption{The Hamiltonian cycle $H_0$}\label{fig3}
\end{figure}
\end{example}

\begin{lemma}\label{lem: sequity gcd = 1 case}
Let $r$ be odd and set $u := \frac{r-1}{2}$.
Then
$ms_3(\ell_i,\ell_{i+u}) \geq 3m+1$ for all $i \in [m-u]$ and
$ms(\ell_i,\ell_{i+u+1}) \geq m$ for all $i \in [m-u-1]$.
\end{lemma}
\begin{proof}
First, we will show that if $i \in [m-u]$,
then $ms_3(\ell_i,\ell_{i+u}) \geq 3m+1$.
Let $\ell = \ell_i \vee_{3m+1} \ell_{i+u}$
and consider a set $E$ of $3m+1$ consecutive edges in $\ell$.
A vertex has degree greater than $3$ in $E$,
only if it has degree two in both
$E \cap E(H_i)$ and $E \cap E(H_j)$.
In particular, the result follows immediately if
there are $m$ or fewer edges in either $E \cap E(H_i)$ or
$E \cap E \left(H_{i+u} \right)$, as the first
$m$ edges of $\ell_a$ form a matching
as do the last $m$ edges of $\ell_a$ for all $a \in [m]$.
So, suppose that there are $2m-l$ edges in $E \cap E(H_i)$
and so $m+1+l$ edges in $E \cap E \left(H_{i+u} \right)$ for some $l \in [m]$.
Let $W_1$ be the vertices of degree at most $1$ in $E \cap E(H_i)$
and $W_2$ be the vertices of degree~$2$ in $E \cap E (H_{i+u})$.
To show that $E$ forms a $(\leq 3)$-regular graph,
it suffices to check that $W_2 \subseteq W_1$.

The first $m$ edges of $\ell_{i+u}$ form a matching
in which $i+u+m$ is the only isolated vertex.
Therefore, the vertices $W_2$ are those incident to edges with labels between
$m$ and $m+l$, excluding $i+u+m$.
In particular,
$W_2 = \{ \infty\} \cup \{ i+\frac{r-1}{2}+rx+\frac{r+1}{2}, i+\frac{r-1}{2}-rx-\frac{r-1}{2} \;:\; x \in [l] \}$,
where $[0] = \emptyset$.
So
$W_2 = \{ \infty\} \cup \{ i+r(x+1), i-rx \;:\; x \in [l] \}
     = \{ \infty\} \cup \left( \{ i+rx, i-rx \;:\; x \in [l+1] \} -\{i-rl\} \right)$,
by re-indexing.
Similarly, the vertices in $W_1$ are those
incident to an edge in $E(H_i)- E$;
these edges are labelled between $0$ and $l$ by $\ell_i$.
Thus,
$W_1 = \{ \infty  \} \cup \{ i+rx' , i-rx' \;:\; x' \in [l+1] \}$.
Comparing $W_1$ and $W_2$ shows that $W_2 \subseteq W_1$.

Finally, we will show that
if $i \in [m-u-1]$, then
$ms(\ell_i,\ell_{i+u+1}) \geq m$.
Let $\ell = \ell_{i} \vee_{m} \ell_{i+u+1}$
and consider a set $E$ of $m$ consecutive edges in $\ell$.
Suppose that there are $m-l$ edges in $E \cap E(H_i)$,
and so $l$ edges in $E \cap E (H_{i+u+1})$ for some non-zero $l \in [m]$.
The first $m$ edges of $\ell_a$ form a matching
as do the last $m$ edges of $\ell_a$ for $a \in [m]$.
Thus,
$E$ does not form a matching only if a vertex is incident to an edge in
$E \cap E \left( H_{i} \right)$ and an edge in $E \cap E(H_{i+u+1})$.
Let $W_0$ be the vertices incident to no edges in $E \cap E \left( H_{i} \right)$
and $W_1$ be the vertices incident to an edge in $E \cap E(H_{i+u+1})$.
To prove that $E$ forms a matching,
it suffices to show that $W_1 \subseteq W_0$.

The last $m$ edges of $\ell_i$ form a matching in which $\infty$
is the only vertex not incident to an edge.
Thus, the members of $W_0$ are the vertices incident to
an edge with a label in between $m+1$ and $m+l$, along with $\infty$.
Hence,
$W_0 = \{ \infty \} \cup \{i+rx+\frac{r+1}{2},i-rx-\frac{r-1}{2} \;:\; x \in [l]\}$.
The members of $W_1$ are the vertices incident
to one of the first $l$ edges of $\ell_{i+u+1}$, so
$W_1 =  \{ \infty, i+\frac{r+1}{2}\} \cup \{i+rx+\frac{r+1}{2},i-rx+\frac{r+1}{2} \;:\; x \in [l]-\{0\}\}$.
Therefore,
$W_1 = \{ \infty \} \cup \left(
\{i+rx+\frac{r+1}{2},i-rx-\frac{r-1}{2} \;:\; x \in [l] \}-\{i-r(l-1)-\frac{r-1}{2}\}\right)
\subseteq W_0$.
This completes the proof.
\end{proof}

\begin{proof}[Proof of Theorem \ref{thm: Matching sequity for r} when $n$ is odd and $\gcd (r,n-1) = 1$]
By Lemma \ref{lem: sequity gcd = 1 case} and
Proposition \ref{prop: 2-regular decomposition odd r both},
using $H_0  ,\ldots, H_{m-1}$ with labellings
$\ell_0  ,\ldots, \ell_{m-1}$, respectively,
we have that $ms_{r}(K_{n}) \geq \frac{rn-1}{2}$.
The reverse inequality follows from
Lemma \ref{lem: Gen matching sequence bound}.
\end{proof}
\section{Proof of Theorem~\ref{thm: Gen cyclic matching sequence odd n} and the remaining cases of Theorem~\ref{thm: Matching sequity for r} }
\label{sec:proof of theorem odd n}
Let $n= 2m+1$, $r \in [2m]-\{0\}$
and let $V_{2m}$ and $H_i$ be as defined in Section~\ref{Decompositions of $K_n$ when $n$ is odd}.
Also, let $\ell_i$ be the labelling of $H_i$ defined as follows:
\begin{align*}
\ell_{i}(\{ \infty, i     \}) &= 0\,, \\
\ell_{i}(\{ \infty, i+m   \}) &= m\,,  \\
\ell_{i}(\{   i+x , i-x   \}) &= x   \quad\quad\;\;\;\:\textrm{for } x \in [m]-\{0\}\,,\\
\ell_{i}(\{   i+x , i-x+1 \}) &= m+x \quad\textrm{for } x \in [m+1]-\{0\}\,.
\end{align*}

\begin{example}
{\rm
When $n = 11$, the Hamiltonian cycle $H_0$ is labelled as follows:
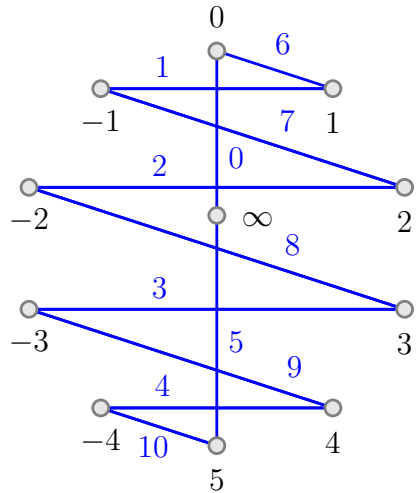
\begin{center}
\begin{tikzpicture}[scale =0.7]
\pgfmathsetlengthmacro\n{360/10}
\pgfmathsetlengthmacro\scfac{2.5cm}
\pgfmathsetlengthmacro\r{\scfac*1.5}
\draw[line width = \scfac*0.015, color = blue]{
(-90 :\r  ) -- node[pos = 0.45, right       ] {5} ( 90 :\r/6)
( 90 :\r/6) -- node[pos = 0.35, right       ] {0} ( 90 :\r)
( 90 :\r  ) -- node[pos = 0.4 , above right ] {6} ( 54 :\r)
};
\draw[line width = \scfac*0.015, color = blue]
{
 (90+1*36:\r)-- node[pos = 0.35, above left ] {1} (90-1*36:\r)
 (90+2*36:\r)-- node[pos = 0.4 , above left ] {2} (90-2*36:\r)
 (90+3*36:\r)-- node[pos = 0.4 , above left ] {3} (90-3*36:\r)
 (90+4*36:\r)-- node[pos = 0.35, above left ] {4} (90-4*36:\r)
 (90+1*36:\r)-- node[pos = 0.55, above right] {7} (90-1*36-36:\r)
 (90+2*36:\r)-- node[pos = 0.65, above right] {8} (90-2*36-36:\r)
 (90+3*36:\r)-- node[pos = 0.81, above right] {9} (90-3*36-36:\r)
 (90+4*36:\r)-- node[pos = 0.45, below      ] {10} (90-4*36-36:\r)
};

\draw[line width = \scfac*0.015, color = blue] \foreach \angle in {1,2,3,4}
{
 (90+\angle*36:\r)-- (90-\angle*36:\r)
 (90+\angle*36:\r)-- (90-\angle*36-36:\r)
};

\draw[line width = \scfac*0.015, color = black!50] \foreach \angle in {1,2,3,4}
{
 (90+\angle*36:\r) node [anchor=south,yshift=-\scfac*0.3, color = black]{$-\angle$}  node[circle, draw, fill=black!10,inner sep=\scfac*0.03, minimum width=\scfac*0.08]{}
 (90-\angle*36:\r) node [anchor=south,yshift=-\scfac*0.3, color = black]{$ \angle$}  node[circle, draw, fill=black!10,inner sep=\scfac*0.03, minimum width=\scfac*0.08]{}
};
\draw[line width = \scfac*0.015, color = black!50]{
(-90 :\r) node [anchor=south,yshift=-\scfac*0.3,color = black]{$5$}
node[circle, draw, fill=black!10,inner sep=\scfac*0.03, minimum width=\scfac*0.08]{}
(90 :\r/6) node [anchor=west,xshift=\scfac*0.07,yshift=-\scfac*0.02, color = black]{$\infty$}
node[circle, draw, fill=black!10,inner sep=\scfac*0.03, minimum width=\scfac*0.08]{}
(90 :\r) node [anchor=north,yshift=\scfac*0.3,color = black]{$0$}
node[circle, draw, fill=black!10,inner sep=\scfac*0.03, minimum width=\scfac*0.08]{}
};

\end{tikzpicture}
\end{center}}
\end{example}

The other $H_i$'s and their labellings are obtained by rotating the above graph.
Therefore,
$\ell_i$ is just the labelling that takes alternating edges of the Hamilton cycle $H_i$
starting from $\{\infty, i\}$.
Thus, it is clear that $ms(\ell_i) = m$ for $i \in [m]$.

\subsection{Proof of Theorem \ref{thm: Matching sequity for r} for when \texorpdfstring{$n$}{n} is odd and \texorpdfstring{$r$}{r} is even }
We will require the following lemma to apply Proposition \ref{prop: 2-regular decomposition both}.
\begin{lemma}\label{lem: Hamiltonian comparison}
For distinct $i,j \in [m]$, $ms_2(\ell_i,\ell_j) \geq 2m+1 -|j-i|$.
\end{lemma}

\begin{proof}
Let $s = 2m+1 -|j-i|$
and consider a set $E$ of $s$ consecutive edges of $\ell = \ell_i \vee_{s} \ell_j$.
As $ms(\ell_a) \geq m$ for all $a \in [m]$,
the last $m$ edges of $\ell_i$ form a matching
as do the first $m$ edges of~$\ell_j$.
Thus,
if both $E \cap E(H_i)$ and $E \cap E(H_j)$ contain fewer than $m+1$ edges,
then $E \cap E(H_i)$ and $E \cap E(H_j)$ both form matchings.
Hence, no vertex would have degree greater than $2$ in~$E$.
So,
consider a set of $s$ consecutive edges $E$ in $\ell$
where either $E \cap E(H_i)$ or $E \cap E(H_j)$
contains at least $m+1$ edges.
Let $a \in \{i,j\}$ be the integer for which $|E \cap E(H_a)| \geq m+1$.
Let $W_2$ be the vertices of degree $2$ in $E \cap E(H_{a})$
and $W_1$ be the vertices of degree $1$ in $E \cap E(H_{a'})$,
where $a' \in \{ i,j\}$ and  $a' \neq a$.
As $ms(\ell_{a'}) \geq m$,
$E \cap E(H_{a'})$ forms a matching.
Thus, there exists a vertex incident to more than two edges in $E$
only if $W_1 \cap W_2 \neq \emptyset$.
Hence, to show that $E$ forms a $(\leq 2)$-regular graph,
it suffices to prove that $W_1 \cap W_2 = \emptyset$.

If $a = i$, then $E \cap E(H_i)$ has $s-l$ edges
and $E \cap E(H_j)$ has $l$ edges for some non-zero $l \in [s-m]$.
The last $m$ edges of $\ell_i$ form a matching in which
$\infty$ is the only vertex not incident to an edge.
Therefore, the vertices in $W_2$ are the vertices incident to an edge
labelled $x$ by $\ell_i$ such that $l+|j-i| \leq x \leq m$,
excluding $\infty$.
In particular, $W_2$ contains $i+m$ and $i+x$ and $i-x$ for $l+|j-i| \leq x \leq m-1$.
The vertices in $W_1$ are the vertices incident
to any of the first $l$ edges of $\ell_j$;
hence,
$W_1$ contains $\infty , j$, and  $j+x'$ and $j-x'$
for $x' \in [l] -\{0\} $.
Then $W_2$ and $W_1$ can be expressed (modulo $2m$) as
\begin{align*}
W_2 &= \{i+l+|j-i| ,\ldots ,i+(m-1),i+m,i-(m-1), \ldots , i-l-|j-i|\}\\
\textrm{and}\quad
W_1 &= \{\infty, j-l+1 ,\ldots , j-1,j,j+1, \ldots , j+l-1\}\,.
\end{align*}
As $i-|j-i| < j+1$ and $i+|j-i| > j-1$, $W_1 \cap W_2 = \emptyset$.
The case in which $a = j$ is similar and we omit the details.
Thus, the $s$ consecutive edges of $E$ form a $(\leq 2)$-regular graph.
\end{proof}

\begin{proof}[Proof of Theorem \ref{thm: Matching sequity for r} when $n$ is odd and $r$ is even]
Let $\alpha$ satisfy the properties in Lemma \ref{lem: General non cyclic ordering}
with $u = \frac{r}{2}$ and $t = m$.
Let $H'_{\alpha(i)} = H_{i}$ and $\ell'_{\alpha(i)} = \ell_{i}$ for all $i \in [m]$.
If $x= \alpha(i) \in [m-u]$,
then $ms_2(\ell'_{x},\ell'_{x+u}) = ms_2(\ell_{i},\ell_{i+1}) \geq 2m$,
by Lemmas~\ref{lem: General non cyclic ordering} and~\ref{lem: Hamiltonian comparison}.
Thus,
applying Proposition \ref{prop: 2-regular decomposition both}
to $H'_0, \ldots ,H'_{m-1}$ with orderings $\ell'_0  ,\ldots, \ell'_{m-1}$,
respectively, yields the inequality $ms_r(K_{n}) \geq \frac{rn}{2}-1$.
The reverse inequality, $ms_r(K_{n}) \leq \frac{rn}{2}-1$, follows
from Lemma \ref{lem: Gen matching sequence bound}.
\end{proof}

\subsection{Proof of Theorem \ref{thm: Matching sequity for r} for when \texorpdfstring{$n$}{n} and \texorpdfstring{$r$}{r} are odd and \texorpdfstring{$r\geq \frac{n-1}{2}$}{r> n-1/2}}
We will require the following two lemmas.

\begin{lemma}\label{lem: Label odd r and n}
For a fixed $\frac{t}{2} \leq u \leq t-1$
there exists an ordering $\alpha_u$ of $[t]$ such that
\begin{itemize}
\item[(1)] $\alpha_u(i+u)   = \alpha_u (i) - 1$ \; for all $i\in [t-u]$;
\item[(2)] $\alpha_u(i+u+1) = \alpha_u (i) + 1$ \; for all $i\in [t-u-1]$.
\end{itemize}
\end{lemma}

\begin{proof}

We show that the ordering $\alpha_u \;:\; [t] \rightarrow [t]$, defined delow, will suffice:
\[
\alpha_u(i) :=
\begin{cases}
  2i+1    &\textrm{ if  } 0   \leq i \leq t-u-1 \\
  i+(t-u) &\textrm{ if  } t-u \leq i \leq u-1 \\
  2(i-u)  &\textrm{ if  } u   \leq i \leq t-1\,.
\end{cases}
\]
It is easy to check that $\alpha_u$
is injective and thus bijective;
$\alpha_u$ is thus an ordering of $[t]$.
For~each integer $i \in [t-u]$,
$\alpha_u(i+u) = 2(i+u - u) = 2i = \alpha_u(i)-1$.
Similarly for $i \in [t-u-1]$,
$\alpha_u(i+u+1) = 2(i+u+1 - u) = 2i+2 = \alpha_u(i)+1$.
\end{proof}

\begin{lemma}\label{lem: Hamiltonian comparison 1 and 3}
If $i < j$, then
$ms(\ell_i, \ell_j) \geq  m+1 - (j-i)$ and $ms_3(\ell_i,\ell_j) \geq 3m+1 -(j-i)$.
If $i > j$,
$ms(\ell_i, \ell_j) \geq  m   - (i-j)$ and $ms_3(\ell_i,\ell_j) \geq 3m+2 -(i-j)$.
\end{lemma}

\begin{proof}
Let $\ell = \ell_{i} \vee_s \ell_j$,
where $s$ is yet to be specified.
First, let
$s = 3m+2 -|j-i| - \epsilon$
where $\epsilon = 0$ if $i > j$ and $\epsilon = 1$
otherwise.
We want to show that $ms_3(\ell) \geq s$.
Consider a set of $s$ consecutive edges $E$
of $\ell$.
For each $a \in [m]$, the first $m$ edges of $\ell_a$ form a matching,
as do the last $m$ edges of $\ell_a$, since $ms(\ell_a) \geq m$.
Thus, if either $E \cap E(H_i)$
or $E \cap E(H_j)$ contain fewer than $m+1$ edges,
then the degree of a vertex cannot be more than $3$ in $E$.
So, it suffices to assume that
$|E \cap E(H_i) | \geq m+1$ and $|E \cap E(H_j) | \geq m+1$.
Suppose that there are $s - (m+l)$ edges in $E \cap E(H_i)$,
and so $m+l$ edges in $E \cap E(H_j)$
for some non-zero $ l \in [s-2m]$.

Let $W_a$ be the vertices incident to two edges in $E \cap E(H_a)$ for $a = i,j$.
To show that $E$ forms a $(\leq 3)$-regular graph,
it suffices to prove that $W_i \cap W_j = \emptyset$,
as $H_i$ and $H_j$ are $2$-regular.
The last $m$ edges of $\ell_i$ form a matching
that covers every vertex except~$\infty$.
Therefore, $W_i$ contains the vertices
that are incident to one of the edges with labels between $|j-i|+l+\epsilon-1$ and $m$, apart from $\infty$.
Thus,
the vertices in $W_i$ are $i+m$ and $i + x$ and $i - x$ for $|j-i|+l+\epsilon-1 \leq x \leq m-1$.
Similarly, the first $m$ edges of $\ell_j$ form a matching
that covers all vertices except $j+m$.
Hence, $W_j$ contains the vertices
that are incident to one of the edges with labels between $m$ and $m+l-1$, except $j+m$.
Therefore,
the vertices in $W_j$ are $\infty$ and $j + x'$ and $j - x' +1$ for $x \in [l] -\{0\}$.
So $W_i \cap W_j$ is clearly empty when $l=1$. In the remaining cases,
we can then express $W_i$ and $W_j$ as follows modulo $2m$:
\begin{align*}
  W_i &= \{i+|j-i|+l+\epsilon-1,\ldots, i+(m-1),i+m,\\
      &\phantom{=}\hspace*{3.2mm}  i-(m-1),\ldots , i-|j-i|-l-\epsilon+1\}\\
      &= \begin{cases}
           \{2i-j+l-1, 2i-j+l ,\ldots ,    j-l+1\} &\;\text{if }\; i>j\\
           \{   j+l,    j+l+1 ,\ldots , 2i-j-l\} &\;\text{if }\; i<j
         \end{cases}\\\textrm{and}\quad
  W_j &= \{\infty, j-l+2 ,\ldots,j-1, j,j+1, \ldots , j+l-1\} \,.
\end{align*}
We see that $W_i \cap W_j = \emptyset$.
Thus, the $s$ consecutive edges of $E$ form a $(\leq 3)$-regular graph.

Now set $\ell := \ell_{i} \vee_{s} \ell_{j}$ with $s = m+1 - |j-i| -\epsilon$
where $\epsilon = 0$ if $i < j$
and $\epsilon = 1$ otherwise.
We~want to show that $ms(\ell) \geq s$.
Consider a set of $s$ consecutive edges $E$
of~$\ell$.
Suppose that there are $s-l$ edges in $E \cap E(H_i)$
and so $l$ edges in $E \cap E(H_j)$ for some non-zero $l \in [s]$.
Let $W_a$ be the vertices incident to an edge in $E \cap E(H_a)$
for $a = i,j$.
As $ms(\ell_b) \geq m$ for all $b \in [m]$,
the last $m$ edges of $\ell_i$ form a matching as do the
first $m$ edges of $\ell_j$.
In particular, $E \cap E(H_i)$ and $E \cap E(H_j)$
are matchings.
Thus, to show that $E$ forms a matching,
it suffices to prove that $W_i \cap W_j = \emptyset$.
The vertices in $W_i$ and $W_j$ are the vertices incident
to one of the last $s-l$ edges of $\ell_i$ and
one of the first $l$ edges of $\ell_j$,
respectively.
Hence,
the vertices in $W_i$ are $i+x$ and  $i-x+1$ for $l+|j-i|+\epsilon \leq x \leq m$,
while
the vertices in $W_j$ are $\infty,j$, and $j+x'$ and $j-x'$ for
$x' \in [l] - \{0\}$.
We can then express $W_i$ and $W_j$ as follows modulo $2m$:
\begin{align*}
  W_i &= \{i+l+|j-i|+\epsilon,\ldots, i+(m-1), i+m,\\
      &\phantom{=}\hspace*{3.2mm}  i-m+1, \ldots, i-|j-i|-\epsilon -l+1\}\\
      &= \begin{cases}
           \{   j+l  ,    j+l+1,\ldots, 2i-j-l+1 \} &\;\text{if }\; i<j\\
           \{2i-j+l+1, 2i-j+l+2,\ldots,    j-l \} &\;\text{if }\; i>j
         \end{cases}\\\textrm{and}\quad
  W_j &= \{\infty, j-l+1,\ldots, j-1, j, j+1,\ldots, j+l-1\}\,.
\end{align*}
We see that $W_i \cap W_j = \emptyset$.
Thus, the $s$ consecutive edges in $E$ form a matching.
\end{proof}

\begin{proof}[Proof of Theorem \ref{thm: Matching sequity for r} for when $n$ and $r$ are odd and $r\geq \frac{n-1}{2}$]
Theorem~\ref{thm: Cyclic matching sequity for r} implies the result for $r =\frac{n-1}{2}$.
So, suppose that $r \geq \frac{n+1}{2}$.
Let $\alpha_{u}$ be a labelling with the properties given in Lemma~\ref{lem: Label odd r and n}
for $u=\frac{r-1}{2}$ and $t = m =\frac{n-1}{2}$.
Set $H'_{i} := H_{\alpha_{u}(i)}$ and $\ell'_i := \ell_{\alpha_u(i)}$ for each $i \in [m]$.
By Lemma \ref{lem: Label odd r and n},
$\ell'_{i+u}=\ell_{\alpha_u(i+u)} = \ell_{\alpha_u(i)-1}$ for each $i \in [m-u]$
and $\ell'_{i+u+1}=\ell_{\alpha_u(i+u+1)} = \ell_{\alpha_u(i)+1}$ for each $i \in [m-u-1]$.
Thus, $ms_3(\ell'_{i},\ell'_{i+u}) \geq 3m+1$ for each $i \in [m-u]$
and $ms(\ell'_{i},\ell'_{i+u+1}) \geq m$ for each $i \in [m-u-1]$,
by Lemma \ref{lem: Hamiltonian comparison 1 and 3}.
Hence, $ms_{r}(K_n) \geq \left\lfloor\frac{rn-1}{2}\right\rfloor$ follows from
Proposition~\ref{prop: 2-regular decomposition odd r both},
using the decomposition
$H_{0}' ,\ldots, H_{m-1}'$ of $K_n$ with orderings $\ell'_0  ,\ldots, \ell'_{m-1}$,
respectively.
Lemma~\ref{lem: Gen matching sequence bound} implies that $ms_{r}(K_n) \leq \left\lfloor\frac{rn-1}{2}\right\rfloor$, completing the proof.
\end{proof}

\subsection{Proof of Theorem \ref{thm: Gen cyclic matching sequence odd n}}
\label{proof of thm on Kn odd n even r}
To prove Theorem~\ref{thm: Gen cyclic matching sequence odd n},
we will need the following ordering of the integers
in $\{l,l+1 \ldots , l+t-1 \}$.
An ordering $\alpha : A \rightarrow [|A|]$
corresponds to a list $a_0 , a_1 \ldots , a_{k-1}$
of the integers of $A$
if $\alpha(a_i) = i$ for all $i$.
Let $t \in [m+1]-\{0\}$ and $l \in [m-t+1]$,
and
let $\alpha_{l,t}$ be the ordering corresponding to
\begin{align*}
    l ,l+2 ,\ldots,l+t-2, l+t-1 ,l+t-3,\ldots, l+1  & \quad\text{if $t$ is even}\\\text{and}\qquad
    l ,l+2 ,\ldots,l+t-1, l+t-2, l+t-4,\ldots, l+1  & \quad\text{if $t$ is odd}\,.
\end{align*}

\begin{proof}[Proof of Theorem~\ref{thm: Gen cyclic matching sequence odd n}]
Set $d := \gcd(\frac{r}{2},m)$ and $c := \frac{m}{d}$.
Let $\alpha$ and $a_{i,j}$ be defined as in Lemma~\ref{lem: General cyclic ordering}
with $u = \frac{r}{2}$ and $t = m$.
Let    $H'_{\alpha(a_{i,j})} :=    H_{\alpha^{-1}_{jc,c}(i)}$
and $\ell'_{\alpha(a_{i,j})} := \ell_{\alpha^{-1}_{jc,c}(i)}$
for all $i \in [c]$ and $j \in [d]$.
This is indeed well-defined, as $\alpha$ is a bijection of $[m]$ and
$\alpha_{jc,c}$ are bijections of the disjoint sets $\{jc , jc+1 ,\ldots, jc+c-1 \}$ for $j \in [d]$.
By Lemma~\ref{lem: General cyclic ordering},
$\alpha(a_{i+1,j}) = \alpha(a_{i,j})+u \pmod{m}$
for all $i \in [c]$ and $j \in [d]$.
By definition, it is clear that
$|\alpha^{-1}_{l,t}(i+1)-\alpha^{-1}_{l,t}(i)| \leq 2$ for all $i\in [t]$
where $i+1$ is reduced modulo $t$.
Thus for $x = \alpha(a_{i,j})$,
Lemma~\ref{lem: Hamiltonian comparison} implies that
\[
  ms_2(\ell'_{x},\ell'_{x+u})
= ms_2\left(\ell'_{\alpha(a_{i,j})},\ell'_{\alpha(a_{i+1,j})}\right)
= ms_2\left(\ell_{\alpha^{-1}_{jc,c}(i)},\ell_{\alpha^{-1}_{jc,c}(i+1)}\right) \geq 2m-1 \,.
\]
By applying Proposition~\ref{prop: 2-regular decomposition both}
to $H'_{0}, \ldots , H'_{m-1}$ with orderings $\ell'_{0}, \ldots, \ell'_{m-1}$,
respectively, we see that
$cms_r(K_{n}) \geq \frac{rn}{2} -2 = \left\lfloor\frac{rn-1}{2}\right\rfloor-1$.
By Lemma~\ref{lem: Gen matching sequence bound}, $cms_r(K_{n})\leq \left\lfloor\frac{rn-1}{2}\right\rfloor$.
\end{proof}

\section{Proof of Theorem~\ref{thm: Cyclic matching sequity for r} for when \texorpdfstring{$r = \frac{n-1}{2}$}{r = n-1/2}}
\label{sec: r matching seq odd r centre case}
We will first prove the case when $r$ is even.
\begin{proof}[Proof of Theorem~\ref{thm: Cyclic matching sequity for r} when $r= \frac{n-1}{2}$ for even $r$]
Let $H_{i}$ be the Hamiltonian
cycle defined
in Subsection \ref{Decompositions of $K_n$ when $n$ is odd}.
Also, let $\ell_i$ be the ordering of $H_i$ defined in
Section \ref{sec:proof of theorem odd n}.
Let $\alpha$  and $a_{i,j}$ be as defined in Lemma \ref{lem: General cyclic ordering}
for $u = \frac{r}{2}$ and $t = m$.
Let $H'_{\alpha(a_{i,j})} := H_{a_{i,j}}$
and
 $\ell'_{\alpha(a_{i,j})} := \ell_{a_{i,j}}$
for $i \in [2]$ and $j \in [\frac{r}{2}] = [\frac{m}{2}]$.
As $\gcd(r,m) = \frac{m}{2}$,
it follows that $\frac{t}{u} = 2$. Thus,
$|a_{i+1,j} -a_{i,j}| = 1$ for all $i,j$.
Lemmas~\ref{lem: General cyclic ordering} and~\ref{lem: Hamiltonian comparison}
therefore imply that,
for $x = \alpha(a_{i,j})$,
\[
    ms_2(\ell'_{x},\ell'_{x+u})
  = ms_2\left(\ell'_{\alpha(a_{i,j})},\ell'_{\alpha(a_{i+1,j})}\right)
  = ms_2\left(\ell_{a_{i,j}},\ell_{a_{i+1,j}} \right)
  \geq 2m
\]
for all $i,j$.
By applying Proposition~\ref{prop: 2-regular decomposition both}
to $H'_{0} , \ldots , H'_{m-1}$  ordered by
$\ell'_{0} , \ldots , \ell'_{m-1}$, respectively,
we see that
$ms_r(K_n) \geq \left\lfloor\frac{rn-1}{2}\right\rfloor$.
The reverse inequality follows from Lemma~\ref{lem: Gen matching sequence bound},
completing the proof.
\end{proof}

Now let $r = m$ be odd and let
$R_i$ be the $2$-regular graph defined in Section~\ref{Decompositions of $K_n$ when $n$ is odd} for $i \in [m]$.
Let $\ell_i$ be the ordering of $R_i$ defined as follows:
\begin{align*}
\ell_i  (\{ v_{\infty}  , v_{i       ,  0}\}) &= 0   \,,\\
\ell_i  (\{ v_{\infty}  , v_{i       ,  1}\}) &= m   \,,\\
\ell_{i}(\{ v_{i+2x  ,x}, v_{i-2x    ,x  }\}) &=   x \quad\quad\:\,\textrm{ for } x \in [m] -\{0\}\,,\\
\ell_{i}(\{ v_{i+2x-1,x}, v_{i-(2x-1),x+1}\}) &= m+x \;\textrm{ for } x \in [m+1]-\{0\}\,.
\end{align*}

\begin{example}
\rm 
When $n=7$, $R_0$ is ordered as
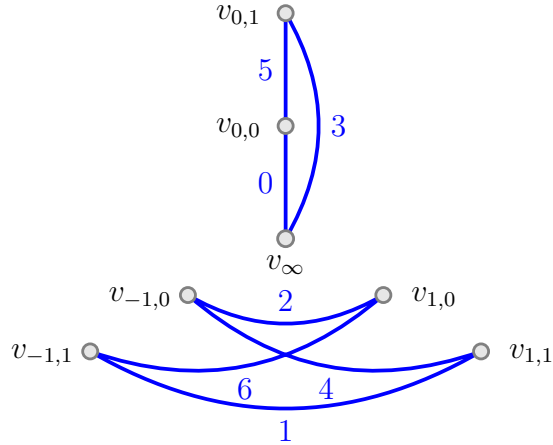
\begin{center}
\begin{tikzpicture}[thick,scale=0.6]
  \pgfmathsetlengthmacro\scfac{2.5cm}
  \pgfmathsetmacro{\sepang}{120}
  \pgfmathtruncatemacro{\c}{3}
\foreach \colora/\mycount in {blue/1}%, red/2/below right,green/3/above right}
\draw[line width = \scfac*0.02, color = \colora]{
%% Adjacent to the centre vertex
[bend right]
(0,0)--node[pos = 0.5, left] {0} ($(90+\sepang*\intcalcMod{\mycount-1}{\c}:1*\scfac)$)
[bend right]
(0,0)to node[pos = 0.5, right] {3} ($(90+\sepang*\intcalcMod{\mycount-1}{\c}:2*\scfac)$)
%% Adjacent vertices of same radii
[bend right]
(90+\sepang*\intcalcMod{\mycount-1}{\c}+\sepang: 2*\scfac) to node[pos = 0.5, below]{1}
(90+\sepang*\intcalcMod{\mycount-1}{\c}-\sepang: 2*\scfac)
[bend right]
(90+\sepang*\intcalcMod{\mycount-1}{\c}+\sepang: 1*\scfac) to node[pos = 0.5, above]{2}
(90+\sepang*\intcalcMod{\mycount-1}{\c}-\sepang: 1*\scfac)
%% Adjacent vertices of same angle
(90+\sepang*\intcalcMod{\mycount-1}{\c} : 1*\scfac)-- node[pos = 0.5, left]{5}
(90+\sepang*\intcalcMod{\mycount-1}{\c} : 2*\scfac)
%%% Remaining edges
[bend right]
(90+\sepang*\intcalcMod{\mycount-1}{\c}+\sepang : 1*\scfac)to node[pos = 0.5, below ]{4}
(90+\sepang*\intcalcMod{\mycount-1}{\c}-\sepang : 2*\scfac)
[bend left]
(90+\sepang*\intcalcMod{\mycount-1}{\c}-\sepang : 1*\scfac)to node[pos = 0.5, below ]{6}
(90+\sepang*\intcalcMod{\mycount-1}{\c}+\sepang : 2*\scfac)
};
\pgfmathsetmacro{\sepangint}{90+\sepang}
\pgfmathsetmacro{\sepangfin}{450-\sepang}
\foreach \r in {1*\scfac,2*\scfac}%,3*\scfac}
\foreach \angle in {90,\sepangint,...,\sepangfin}
\draw[line width = \scfac*0.015, color = black!50]{
  (\angle : \r) node[circle, draw, fill=black!10,inner sep=\scfac*0.015, minimum width=\scfac*0.08]{}};
\draw[line width = \scfac*0.015, color = black!50]{
  (0 : 0) node[circle, draw, fill=black!10,inner sep=\scfac*0.03, minimum width=\scfac*0.08]{}
  (90 : \scfac*0.1) node [xshift=-\scfac*0,yshift=-\scfac*0.2,color = black]{$v_{\infty}$}};
\draw[line width = \scfac*0.05, color = black!30]{
  (  90 : 1*\scfac) node [anchor=west,xshift=-\scfac*0.45,yshift=-\scfac*0.02,color = black]{$v_{ 0,0}$}
  ( -30 : 1*\scfac) node [anchor=west,xshift= \scfac*0.07,yshift=-\scfac*0.02,color = black]{$v_{ 1,0}$}
  (-150 : 1*\scfac) node [anchor=west,xshift=-\scfac*0.5 ,yshift=-\scfac*0.02,color = black]{$v_{-1,0}$}
  (  90 : 2*\scfac) node [anchor=west,xshift=-\scfac*0.45,yshift=-\scfac*0.02,color = black]{$v_{ 0,1}$}
  ( -30 : 2*\scfac) node [anchor=west,xshift= \scfac*0.07,yshift=-\scfac*0.02,color = black]{$v_{ 1,1}$}
  (-150 : 2*\scfac) node [anchor=west,xshift=-\scfac*0.5 ,yshift=-\scfac*0.02,color = black]{$v_{-1,1}$}};
\end{tikzpicture}
\end{center}
\end{example}

For each $i$, set $R_{i} := R_{i'}$ and $\ell_i := \ell_{i'}$ where $i' \equiv i \pmod{m}$.
It is easy to check that $\ell_i$ is indeed a valid ordering of $R_i$
and that the first $m$ edges of $\ell_i$ form a matching
as do the last $m$ edges of $\ell_i$.

\begin{lemma}\label{lem: 1 and 3 seq centre case}
For all $i$, $ms_3(\ell_i,\ell_{i+1})\geq 3m+1$ and $ms(\ell_i,\ell_{i-1}) \geq m$.
\end{lemma}
\begin{proof}
First, we will show that $ms_3(\ell_i,\ell_{i+1})\geq  3m+1$ for all $i$.
Set $\ell := \ell_{i} \vee_{3m+1} \ell_{i+1}$
and consider a set of $3m+1$  consecutive edges $E$ of~$\ell$.
A vertex $v$ has degree more than $3$ in $E$ only if
$v$ has degree two in both $E \cap E(R_{i})$ and $E \cap E(R_{i+1})$.
The first $m$ edges of $\ell_j$ form  a matching
as do the last $m$ edges of $\ell_j$ for all $j$.
Therefore,
if there are $m$ or fewer edges in either $E \cap E(R_{i})$ or $E \cap E(R_{i+1})$,
then $E$ forms a $(\leq 3)$-regular graph.
So, suppose that there are $2m+1-l$ edges in $E \cap E(R_i)$
and so $m+l$ edges in $E \cap E(R_{i+1})$ for some non-zero $l \in [m+1]$.
Let $W_2$ be the vertices of degree $2$ in $E \cap E(R_{i+1})$
and $W_1$ be the vertices of degree at most $1$ in $E \cap E(R_{i})$.
To show that $E$ forms a $(\leq 3)$-regular graph,
it suffices to show that $W_2 \subseteq W_1$.

As $R_i$ is $2$-regular and the first $m$ edges of $\ell_i$ form a matching,
$W_1$ contains the vertices incident to one of the first $l$ edges of $\ell_i$.
Thus,
$W_1 = \{ v_{\infty} \} \cup \{ v_{,i+2x,x},v_{i-2x,x} \;:\; x \in[l] \}$.
The first $m$ edges of $\ell_{i+1}$ form a matching which covers
every vertex except $v_{i,1}$.
Therefore, $W_2$
the vertices incident to an edge with label between $m$ and $m+l-1$,
excluding $v_{i,1}$:
$W_2 = \{ v_{\infty} \} \cup \left\{ v_{i+1+2x'-1,x'},v_{i+1-(2x'-1),x'+1} \;:\;x' \in [l]-\{0\}\right\}$.
By simplifying and re-indexing,
we see that
$W_2 = \{ v_{\infty}\} \cup
\left(\left\{ v_{i+2x',x'},v_{i-2x',x'} \;:\; x' \in [l] \right\}-\{ v_{i-2(l-1),l-1}\}\right) \subseteq W_1$.
Therefore, $E$ forms a $(\leq 3)$-regular graph and $ms_3(\ell_i , \ell_{i+1}) \geq 3m+1$.

Finally, we show that $ms(\ell_i,\ell_{i-1}) \geq m$ for all $i$.
Set $\ell := \ell_i \vee_m \ell_{i-1}$
and consider a set of $m $ consecutive edges $E$ of $\ell$.
Suppose that there are $m-l$ edges in $E \cap E(R_i)$
and thus $l$ edges in $E \cap E(R_{i-1})$ for some non-zero $l\in [m]$.
Let $W_1$ be the vertices incident to an edge in $E \cap E(R_{i-1})$
and $W_0$ be the vertices not incident to any edge in $E \cap E(R_{i})$.
The last $m$ edges of $\ell_{i}$ form a matching as do the first $m$ edges of $\ell_{i-1}$.
Thus,
a vertex $v$ is incident to 2 or more edges of $E$ only if
$v \in W_1$ and $v$ is incident to an edge in $E \cap E(R_{i})$.
Therefore, it suffices to show that $W_1 \subseteq W_0$.

The last $m$ edges of $\ell_i$ form a matching in which $v_\infty$ is the only isolated vertex.
Therefore, $W_0$~contains the vertex $v_\infty$ along with
the vertices incident to an edge with label between $m+1$ and $m+l$;
that is,
$W_0 = \{ v_{\infty} \} \cup \left\{ v_{i+2x-1,x}, v_{i-(2x-1),x+1} \;:\; x \in[l+1]-\{0\}\right\}$.
By re-indexing, we see that
$W_0 =\! \{ v_{\infty} ,v_{i-1,0},v_{i-1+2l,l}  \} \cup
\left\{ v_{i-1+2x,x}, v_{i-1-2x,x}
: x \!\in \![l]- \{0\} \right\}$.
Now, $W_1$ contains the vertices that are incident to one of the first $l$ edges of $\ell_{i-1}$.
In other words,
$W_1 = \{ v_{\infty},v_{i-1,0} \} \cup
\left\{ v_{i-1+2x',x'} , v_{i-1-2x',x'} \;:\; x' \in [l]- \{0\} \right\} \subseteq W_0$.
Hence, $E$ forms a matching and $ms(\ell_i , \ell_{i-1}) \geq m$, as required.
\end{proof}

\begin{proof}[Proof of Theorem~\ref{thm: Cyclic matching sequity for r} when $r= \frac{n-1}{2}$ for odd $r$]
Set $u := \frac{r-1}{2} = \frac{m-1}{2}$ and
let $\beta \;:\; [m] \rightarrow [m]$
be the function defined by $\beta(i) := iu^{-1} \pmod{m}$ for all $i \in [m]$.
The function $\beta$ is clearly a bijection.
Set    $R'_{x} :=    R_{\beta(x)}$
and $\ell'_{x} := \ell_{\beta(x)}$ for $x \in [m]$.
For any $i$, $\beta (i+u) \equiv (i+u)u^{-1} \equiv \beta(i)+1 \pmod{m}$.
Also, as $u^{-1} \equiv -2 \pmod{m}$, we see that
\[
\beta (i+u+1) \equiv -2(i+u+1) \equiv -2i-2u-2 \equiv -2i-1 \equiv \beta(i)-1 \pmod m \,.
\]
Lemma~\ref{lem: 1 and 3 seq centre case} thus implies that, for any $x\in [m]$,
$ms(\ell'_{x},\ell'_{x+u+1}) = ms(\ell_{\beta(x)},\ell_{\beta(x)-1})\geq m$
and
$ms_3(\ell'_{x},\ell'_{x+u}) = ms(\ell_{\beta(x)},\ell_{\beta(x)+1}) \geq 3m+1$.
By applying Proposition~\ref{prop: 2-regular decomposition odd r both}
to $R'_0, \ldots, R'_{m-1}$
ordered by $\ell'_0, \ldots, \ell'_{m-1}$,
respectively,
we see that $cms_r(K_n) \geq \left\lfloor\frac{rn-1}{2}\right\rfloor$.
By Lemma~\ref{lem: Gen matching sequence bound},
$cms_r(K_n) \leq \left\lfloor\frac{rn-1}{2}\right\rfloor$,
and the result follows.
\end{proof}

\section{General conditions and the proof of Theorem \ref{thm: K_n comp equiv}}\label{sec: Gen}
In the process of proving Theorem \ref{thm: K_n comp equiv},
we develop some notions of sequencibility
where an arbitrary condition is placed on the subgraphs
formed by consecutive edges.
We express such a condition by letting $\mathcal{C}$ be
an arbitrary family of graphs on a fixed set of vertices $V$ with some fixed vertex labelling.

A ordering $\ell$ of some graph is {\em cyclically  $(s,\mathcal{C})$-sequenceable}
if all $s$ cyclically consecutive edges in $\ell$ form a graph in $\mathcal{C}$.
A graph $G$ is {\em cyclically $(s,\mathcal{C})$-sequenceable}
if there exists a cyclically $(s,\mathcal{C})$-sequenceable ordering $\ell$ of~$G$.
Note that $s$ is not maximised here:
for an arbitrary set of conditions $\mathcal{C}$,
maximising $s$ may be trivial or otherwise not of interest.
For a graph $G = (V,E)$, let
\[
  \mathcal{C}^{\complement_{G}} :=
  \{(V, E(C) \Delta E(G))  \; : \; C \in \mathcal{C}  \}\,,
\]
where $E(C) \Delta E(G)$ is the symmetric difference of $E(C)$ and $E(G)$.

\begin{lemma}\label{lem: comp cond}
Let $\mathcal{C}$ be a set of conditions on vertex-labelled graphs;
let $G$ be a graph, and let $s$ be an integer.
Then for an ordering $\ell$ of $G$,
$\ell$ is cyclically $(s,\mathcal{C})$-sequenceable
if and only if
$\ell$ is cyclically $(|E(G)|-s,\mathcal{C}^{\complement_{G}})$-sequenceable.
\end{lemma}

\begin{proof}
Let $k = |E(G)|$ and let $\ell$ be an ordering of $G$ which is cyclically $(s,\mathcal{C})$-sequenceable.
Also, let $e_i$ be the edge of $G$ labelled $i$ by $\ell$.
Consider a set of $(k-s)$-cyclically consecutive edges $E$ in $\ell$,
namely $e_{j}, e_{j+1}, \ldots , e_{j+k-s-1}$ for some $j \in [k]$,
where the subscripts are taken modulo $k$.
The edges of $G$ not in $E$ are
$e_{j+k-s}, e_{j+k-s+1}, \ldots , e_{j-1}$,
and they are in this order in~$\ell$.
By assumption, the $s$ edges of $E(G)- E$ form a graph in $\mathcal{C}$.
Thus, the $(k-s)$-cyclically consecutive edges
$e_{j}, e_{j+1}, \ldots , e_{j+k-s-1}$
must form a member of $\mathcal{C}^{\complement_{G}}$.
Hence, $\ell$ is cyclically $(k-s,\mathcal{C}^{\complement_{G}})$-sequenceable.
If $\ell$ is cyclically $(|E(G)|-s,\mathcal{C}^{\complement_{G}})$-sequenceable,
then it follows,
from the above argument and the identities $(\mathcal{C}^{\complement_{G}})^{\complement_{G}} = \mathcal{C}$ and $k-(k-s)=s$,
that $\ell$ is cyclically $(s,\mathcal{C})$-sequenceable.
\end{proof}

\begin{proof}[Proof of Theorem \ref{thm: K_n comp equiv}]
Let $s_a = \bigl\lfloor\frac{an-1}{2}\bigr\rfloor$ for each $a \in [n-1]$.
Also, let $\mathcal{C}_r$ be the set of all vertex-labelled $(\leq r)$-regular graphs on
$n$ vertices.
Suppose that $cms_r(\ell) = s_r$ and, in particular,
suppose that $\ell$ is cyclically $(s_r,\mathcal{C}_r)$-sequenceable for some ordering of $K_n$.
Then, by Lemma \ref{lem: comp cond},
$\ell$ is $\bigl(\frac{n(n-1)}{2}-s_r,\mathcal{C}_r^{\complement_{K_n}}\bigr)$-sequenceable.
Set $s' := \frac{n(n-1)}{2}-s_r = s_{n-1-r}+1$.
Then $\mathcal{C}_r^{\complement_{K_n}}$ is
the family of all vertex-labelled subgraphs of $K_n$ whose vertices each have degree at least $n-1-r$.
The minimum number of edges in a member of $\mathcal{C}_r^{\complement_{K_n}}$ is~$s'$.
Also, any member of $\mathcal{C}_r^{\complement_{K_n}}$ with $s'$ edges must be a graph
in which each vertex has degree $n-1-r$ except one vertex which has degree $n-r$.

Consider $s' +1$ cyclically consecutive edges
$e_0 , \ldots , e_{s'}$ in $\ell$.
Since $\ell$ is $(s'\!,\mathcal{C}_r^{\complement_{K_n}})$-sequenceable,
the edges of each of $E_0 := \{ e_0,\ldots , e_{s' - 1}\}$
and $E_1 := \{e_{1},\ldots , e_{s'}\}$
form a member of $\mathcal{C}_r^{\complement_{K_n}}$.
Also as $s' < \frac{n(n-1)}{2}$,
$e_0 \neq e_{s' }$.
I claim that this ensures that
the edges $E' := \{ e_1,\ldots, e_{s' - 1}\}$ form a $(\leq n-1-r)$-regular graph.
Assume otherwise; then some vertex $v$ is incident to at least $n-r$ of the edges in $E'$.
Let $v_0$ and $v_1$ be the endpoints of $e_0$.
The edges of $E_0$ must form a graph in $\mathcal{C}_r^{C_{K_n}}$,
i.e., a graph whose vertices each has degree $n-1-r$ except one which has degree $n-r$.
As $v$ is incident to at least $n-r$ of the edges in $E' \subseteq E_0$,
$v$ must be incident to exactly $n-r$ of the edges in $E_0$.
In particular, $v$ must be distinct from $v_0$ and $v_1$.
So, $v_0$ and $v_1$ are each incident to $n-1-r$ of the edges in $E_0$.
However, this means that $v_0$ and $v_1$ are each incident to $n-2-r$ of the edges in $E'$.
Thus, either $v_0$ or $v_1$ is incident to only $n-2-r$ of the edges in $E_1$,
since $e_{s'} \neq e_0$.
Therefore, the graph formed by the edges of $E_1$ is not in $\mathcal{C}_r^{\complement_{K_n}}$,
a contradiction.

Any set $E'$ of $s' - 1 = s_{n-1-r}$ cyclically consecutive edges in $\ell$
is a consecutive subsequence of some $s' + 1$ cyclically consecutive edges
in $\ell$ of the form $e_0 \vee L_{\ell}(E') \vee e_{s'}$.
Thus, by the above argument, every vertex must have degree at most $n-1-r$ in $E'$.
Hence, $cms(\ell) \geq s_{n-1-r}$
and, therefore, $cms_{n-1-r}(K_n) \geq s_{n-1-r}$.
By Lemma~\ref{lem: Gen matching sequence bound},
$cms_{n-1-r}(K_n) = s_{n-1-r}$.
The reverse direction,
namely that $cms_{n-1-r}(K_n) = s_{n-1-r}$ implies $cms_{r}(K_n) = s_{r}$,
follows by applying the above argument with $r$ replaced by $n-1-r$.
\end{proof}

Note that a similar result to Lemma~\ref{lem: comp cond}
could be can proved for non-cyclic sequences.
However,
the notion of sequencibility would have to be generalised
to allow partially cyclical sequences.
A result analogous to Theorem~\ref{thm: K_n comp equiv} follows from a similar proof,
but there is not an equivalence between
$ms_r(K_n) = \frac{rn-1}{2}$ and $ms_{n-1-r}(K_n) = \frac{(n-1-r)n-1}{2}$
for odd $r$ and $n$.

\section{Concluding remarks}\label{sec: Con}

When $r$ is even and $n$ is odd,
we expect that $cms_r(K_n) = \bigl\lfloor\frac{rn-1}{2}\bigr\rfloor$.
Theorem \ref{thm: Cyclic matching sequity for r}
confirms this for even $r = \frac{n-1}{2}$.
By computer search, we were able to find the following two orderings for $K_7$
that show that $cms_2(K_7) = 6 = \bigl\lfloor\frac{2n-1}{2}\bigr\rfloor$
and $cms_4(K_7) = 13 = \bigl\lfloor\frac{4n-1}{2}\bigr\rfloor$, respectively.
The orderings are represented by the sequence of edges that has corresponding ordering value sequence $0,\ldots,20$.
\begin{align*}
&\{\infty,0 \}, \{ 1,2\},\{ 3,-2\},\{ 3,-1\},\{ 1,-1\},\{ \infty,-2\},\{ 0,2\}, \\
&\{ 0,1\},\{ 2,3\},\{ \infty,-1\},\{ -2,-1\},\{ 1,3\},\{ \infty,2\},\{ 0,-2\},\\
&\{ 0,3\},\{ \infty,1\},\{2,-1\},\{ 1,-2\},\{ \infty,3\},\{ 0,-1\},\{ 2,-2\}\,;\\[3mm]
&\{\infty,0\}, \{ \infty,1\},\{ 0,2\},\{1,3\},\{ 2,-2\},\{ 0,-1\},\{ -1,-2\},  \\
&\{ \infty,3\},\{ 1,2\},\{ \infty,-2\},\{0,3\},\{ 1,-1\},\{ 2,3\},\{ 0,-2\},\\
&\{ \infty,-1\},\{ 0,1\},\{\infty,2\},\{3,-2\},\{2,-1\},\{ 1,-2\},\{ 3,-1 \}\,.
\end{align*}
We also found an ordering for $K_9$, 
showing that $cms_2(K_9) = 8 = \bigl\lfloor\frac{2n-1}{2}\bigr\rfloor$,
as given below.
\begin{align*}
&\{\infty,0 \}, \{ \infty,1\},\{ 0,2\},\{ 1,3\},\{ 2,4\},
\{ 3,-3\},\{ 4,-2\},\{ -1,-3\},\{ \infty,-2\}  \\
&\{ 0,1\},\{ \infty,2\},\{ 1,-1\},\{ 2,3\},\{ 0,-3\},
\{ 3 , 4\},\{ -3,-2\},\{ \infty,4\},\{ 0,-1\}, \\
&\{ 1,-2\},\{ 2,-1\},\{\infty,3\},\{ 2,-3\},\{ 0,4\},
\{ 1,-3\},\{ 3,-2\},\{ 4,-1 \},\{ 0,-2\},\\
&\{ \infty, -1 \},\{ 1,2\},\{\infty , -3\},\{ 0, 3\},\{ 1,4\},
\{ 2,-2\},\{ 3,-1\},\{ 4,-3\},\{ -1,-2\} \,.
\end{align*}

When $r$ and $n$ are both odd,
Theorem \ref{thm: Cyclic matching sequity for r} implies that
$cms_r(K_{n}) = \bigl\lfloor\frac{rn-1}{2}\bigr\rfloor$ for $r = \frac{n-1}{2}$.
If there are other cases for which $cms_{r}(K_{n}) = \left\lfloor\frac{rn-1}{2}\right\rfloor$
with $r$ and $n$ odd, then, by Theorem~\ref{thm: K_n comp equiv},
the condition for which $cms(K_n) = \left\lfloor\frac{rn-1}{2}\right\rfloor$ holds
must be invariant under replacing $r$ with $n-1-r$.

\begin{proposition}\label{prop: 1 to r}
For a graph $G$ and integers $r_1,r_2$,
\[
 ms_{r_1 r_2} (G) \geq r_2 \, ms_{r_1}(G) \quad \textrm{and} \quad
cms_{r_1 r_2} (G) \geq r_2 \, cms_{r_1}(G) \,.
\]
\end{proposition}
Note that this proposition and Theorems~\ref{thm:Matching sequence} and~\ref{thm:Cyclic matching sequence}
together imply that $ms_r(K_n) \geq r\frac{n-1}{2}$ and $cms_r(K_n) \geq r\frac{n-3}{2}$,
respectively,
when $n$ is odd.
\begin{proof}
Let $s = cms_{r_1}(G)$ and
let $\ell$ be a labelling of $G$ for which $cms_{r_1}(\ell) = cms_{r_1}(G)$.
Any set of $r_2 s$ cyclically consecutive edges $E$ of $\ell$
are just $r_2$ sets of $s$ cyclically consecutive edges of~$\ell$
and in each set, every vertex has degree at most $r_1$.
Thus, every vertex has degree at most $r_1 r_2$ in $E$
and, hence, $cms_{r_1r_2}(G) \geq cms_{r_1r_2}(\ell) \geq r_2 s$.
The non-cyclic case is similar.
\end{proof}

A hypergraph is a pair $(V,E)$
where $V$ is a set and $E$ is a family of subsets of $V$.
A $k$-graph is a hypergraph $(V,E)$ for which each member $e\in E$ has cardinality $|e| = k$.
For instance, each graph is a $2$-graph.
The notion of matching sequencibility naturally extends to hypergraphs,
as do Proposition \ref{prop: 1 to r} and the propositions of Section~\ref{Preliminaries},
using analogous proofs.
For example the natural
hypergraph analogue of Proposition~\ref{prop: Matching decomposition both} is as follows.
\begin{proposition}
\label{prop: Matching decomposition both-hypergraph}
Let $\mathcal{H}$ be a hypergraph that decomposes into matchings $\mathcal{M}_0 ,\ldots , \mathcal{M}_{t-1}$, each with $n$ edges
and orderings $\ell_0 , \ldots , \ell_{t-1}$, respectively.
Suppose, for some $\epsilon \in [n]$ and $r < \Delta(\mathcal{H})$,
that $ms(\ell_i ,\ell_{i+r}) \geq n-\epsilon$ for all $i\in [t-r]$.
Then $ms_r(\mathcal{H}) \geq rn-\epsilon$,
and if $ms(\ell_i ,\ell_{i+r}) \geq n-\epsilon$ for all $i\in [t]$,
then $cms_r(\mathcal{H}) \geq rn-\epsilon$.
\end{proposition}
The natural analogue of $K_n$ for $k$-graphs is the complete $k$-graph on $n$ vertices,
denoted $\mathcal{K}^k_n$, whose edges are all the vertex subsets of size $k$.
Katona~\cite{MR2181045} proved that $cms(\mathcal{K}^k_n) \geq \lfloor\frac{n}{k^2}\rfloor$ for sufficiently large $n$,
under the assumption that a particular conjecture holds.
Katona~\cite{MR2181045} also conjectured that $cms(\mathcal{K}^k_n) \geq \lfloor\frac{n}{k}\rfloor-1$.
By similar reasoning to Lemma~\ref{lem: Gen matching sequence bound},
$ms_r(\mathcal{G})\leq \bigl\lfloor\frac{rn-1}{k}\bigr\rfloor$ for any $k$-graph~$\mathcal{G}$,
which is not a $(\leq r)$-regular $k$-graph.
This leads us to conjecture that
$ms_r(\mathcal{K}^k_n)= \bigl\lfloor\frac{rn-1}{k}\bigr\rfloor$
and $\bigl\lfloor\frac{rn-1}{k}\bigr\rfloor-1 \leq cms_r(\mathcal{K}^k_n) \leq \bigl\lfloor\frac{rn-1}{k}\bigr\rfloor$
for all $r$, $n$, and $k$,
and to expect that $cms_r(\mathcal{K}^k_n)$ can attain both bounds.

We prove a result similar to Katona's, Theorem~\ref{thm:Katonaplus} below,
for the special case in which~$k \mid n$.

\begin{theorem}
\label{thm:Katonaplus}
Let $k \mid n$, $a$ be the largest integer such that $\frac{n}{k}-(a-1)k>0$
and $b$ be the largest integer such that $\frac{n}{k}-(b-1)(k+1)>0$.
Then, for $r < \Delta(\mathcal{K}^k_n)$,
\[
 ms_r(\mathcal{K}^k_n) \geq (r-1)\frac{n}{k} + a \qquad \textrm{and} \qquad
cms_r(\mathcal{K}^k_n) \geq (r-1)\frac{n}{k} + b \,.
\]
\end{theorem}

To prove this theorem,
we use Baranyai's Theorem~\cite{MR0416986} which states that if $k \mid n$,
then $\mathcal{K}^k_n$ has a complete matching decomposition.
Note that in such a decomposition, each matching has size $\frac{n}{k}$.
\begin{proof}[Proof of Theorem \ref{thm:Katonaplus}]
By Baranyai's Theorem, there is a matching decomposition of $\mathcal{K}^k_n$,
say, $\mathcal{M}_{0},\ldots, \mathcal{M}_{N-1}$ where $N = \frac{k}{n}\binom{n}{k}$.
Set $d := \gcd(r,N)$ and $c := \frac{N}{d}$,
and arbitrarily re-index the matching decomposition as $\mathcal{M}_{i,j}$
for $i \in [c]$ and $j \in [d]$.
We construct orderings $\ell_{i,j}$ of $\mathcal{M}_{i,j}$ for $i \in [c]$ and $j \in [d]$
such that $ms(\ell_{i,j},\ell_{i+1,j}) \geq b$ for all $i \in [c]$.
Choose an arbitrary ordering $\ell_{0,j}$ of $\mathcal{M}_{0,j}$.
Suppose, by induction on $i$ for a fixed $j$,
that $\ell_{0,j},\ldots, \ell_{i,j}$ have been constructed,
that the first~$l$ edges of $\ell_{i+1,j}$ are $e_{0},\ldots, e_{l-1}$,
and that any $b$ consecutive edges of $L_{\ell_{i,j}}(\mathcal{M}_{i,j}), e_{0},\ldots, e_{l-1}$ form a matching.
Let $e'_{0}, \ldots, e'_{b-2}$ be the last $b-1$ edges of $\ell_{i,j}$.
If $l \leq b-2$, then
there are at least $\frac{n}{k}-(b-l-1)k-l >0$ edges in $\mathcal{M}_{i+1,j}-\{e_{0} , \ldots , e_{l-1} \}$
which do not share a common vertex with any of the edges $e'_{l} , \ldots , e'_{b-2}$.
Thus, we can choose an edge $e_{l}$ in $\mathcal{M}_{i+1,j}-\{e_{0} , \ldots , e_{l-1} \}$
such that $e'_{l} , \ldots , e'_{b-2} , e_{0} , \ldots, e_{l-1},e_l$ forms a matching
and let $\ell_{i+1,j}(e_l) = l$.
If $l > b-2$, then for an arbitrary edge $e_{l}$ in $\mathcal{M}_{i+,j}-\{e_{0} , \ldots , e_{l-1} \}$,
let $\ell_{i+1,j}(e_l) = l$.
Now let $e_{0} , \ldots , e_{b-2}$ be the last $b-1$ edges of $\ell_{0,j}$.
We are free to permute the ordering $\ell_{0,j}$ and maintain the identity
$ms(\ell_{0,j},\ell_{1,j})\geq b$
so long as the labels of $e_{0}, \ldots , e_{b-2}$ remain unchanged.
Therefore, we can apply a similar argument to the above to permute $\ell_{0,j}$
in such a way that ensures that $ms(\ell_{c-1,j},\ell_{0,j})\geq b$,
since $\frac{n}{k}-(b-1)-(b-l-1)k-l >0$ for all $l \in [b-1]$.
Set $\mathcal{M}'_{\alpha(a_{i,j})} := \mathcal{M}_{i,j}$
and $\ell'_{\alpha(a_{i,j})} := \ell_{i,j}$
where $\alpha$ and~$a_{i,j}$ are defined as in Lemma~\ref{lem: General cyclic ordering}
for $u = r$ and $t = N$.
Proposition~\ref{prop: Matching decomposition both-hypergraph} and Lemma~\ref{lem: General cyclic ordering} yields the second inequality
using the matchings $\mathcal{M}'_{0} , \ldots, \mathcal{M}'_{N-1}$ ordered by
$\ell'_{0} , \ldots, \ell'_{N-1}$, respectively.
The non-cyclic case is similar and, therefore, omitted.
\end{proof}

K\"{u}hn and Osthus~\cite{MR3213309} offer an alternate decomposition of
$\mathcal{K}_n^{k}$
than those given by Baranyai's Theorem,
into {\em Berge cycles} which broadly generalise cycles in graphs.
Their decomposition would however not be likely to be useful for proving a matching sequencibility result.

\subsection*{Acknowledgements}
I would like to thank Thomas Britz for helping improve the presentation
of the article and helpful comments.

\end{document}